\documentclass[10pt]{amsart}
\usepackage{graphicx}
\usepackage[a4paper,top=3cm,bottom=2cm,left=3cm,right=3cm,marginparwidth=1.75cm]{geometry}
 \usepackage[parfill]{parskip}
 \usepackage{amsmath}
\usepackage{amsthm}
\usepackage{enumerate}
\usepackage{thm-restate}
\usepackage{cleveref}
\usepackage{comment}
\usepackage{amssymb}
\usepackage[T1]{fontenc}
\usepackage[colorinlistoftodos]{todonotes}

\newtheorem{theorem}{Theorem}[section]
\newtheorem{lemma}[theorem]{Lemma}
\newtheorem{proposition}[theorem]{Proposition}
\newtheorem{corollary}[theorem]{Corollary}

\theoremstyle{definition}

\newtheorem{definition}[theorem]{Definition}

\newtheorem{remark}[theorem]{Remark}

\newtheorem{claim}[theorem]{Claim}

\def\N{\Bbb N}
\def\R{\Bbb R}

\title{On distinct angles in the plane}
\author{Sergei V. Konyagin, Jonathan Passant, Misha Rudnev}
\thanks{Jonathan Passant is supported by the Heilbronn Institute.}
\address{Sergei V. Konyagin, Steklov Mathematical Institute, Gubkina Str. 8, Moscow  119991, Russia}
\email{konyagin23@mi-ras.ru} 

\address{Jonathan Passant, Department of Mathematics, University of Bristol,
 Bristol BS8 1TW, United Kingdom}
\email{jonathanpassant@gmail.com}

\address{Misha Rudnev, Department of Mathematics, University of Bristol,
 Bristol BS8 1TW, United Kingdom}
\email{m.rudnev@bristol.ac.uk}

\begin{document}

\maketitle

\begin{abstract} 
We prove that if $N$ points lie in convex position in the plane then they determine $\Omega(N^{5/4})$ distinct angles, provided that the points do not lie on a common circle. 

This is derived from a more general claim that if $N$ points in the convex position in the real plane determine $KN$ distinct angles, then $K=\Omega(N^{1/4})$ or $\Omega(N/K)$ points are co-circular. 

The proof makes use of the implicit order one can give to points in convex position and relies on a slightly more general order assumption. The assumption enables one to reduce the issue to counting incidences between points and a multiset of cubic curves, with special attention being paid to the case when the curves are reducible.
\end{abstract} 

\section{Introduction}
How many distinct angles, formed by triples of points does a set of $N$ non-collinear points in $\mathbb R^2$ define? 
Two obvious examples, giving, up to a constant multiplier, some $N$ distinct angles are as follows.  One can take a right $N$-gon together with its centre. Or one can take two points, symmetric relative to a line, supporting the rest of the points, so that the angles with the base on the line and vertex at one of the former two points are in an arithmetic progression. See Figure \ref{f:examples}.

We will use the standard notation $O$ and $\Omega$ for upper and lower bounds up to absolute constants, along with the symbols $\ll,\gg$, respectively. The symbol $\Theta$ means both $O$ and $\Omega$. If an estimate using these symbols contains a parameter $\epsilon$, the implicit constant will  also depend on $\epsilon$.

One can modify these examples by taking generalised progressions to vary the number of distinct angles between $\Theta(N)$ and $\Theta(N^2)$.

However, if all the points but one lie on a circle, and the latter point is not the centre, it may well be the number of distinct angles is $\Omega(N^{2-\epsilon})$ for any $\epsilon >0$ (or even with $\epsilon =0$) the constant in the $\Omega$-symbol (possibly) depending on $\epsilon$. We present the lower bound $\Omega\left( N^{\frac{13}{10}-\epsilon}\right)$ as an easy consequence of the known estimates on convexity versus sumsets, formulated as Theorem \ref{th:conv} (see the proof of Theorem \ref{t:unc}).

\begin{figure}[h!]
    \begin{minipage}{0.49\textwidth}\centering
\includegraphics[scale=0.4]{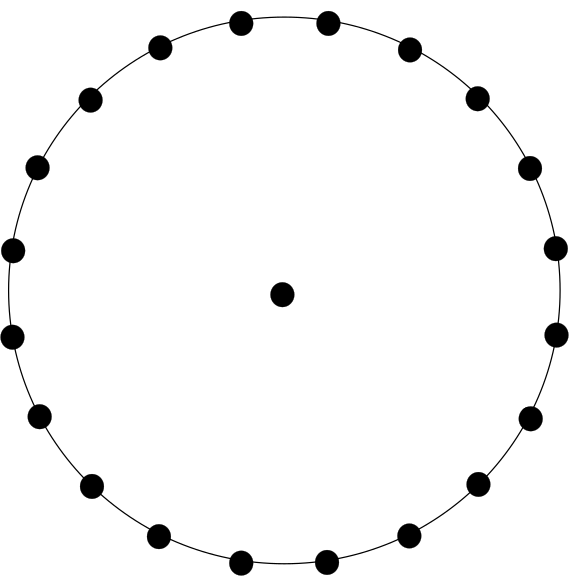} 
    \end{minipage}
    \begin{minipage}{0.49\textwidth}\centering
       \includegraphics[scale=0.3]{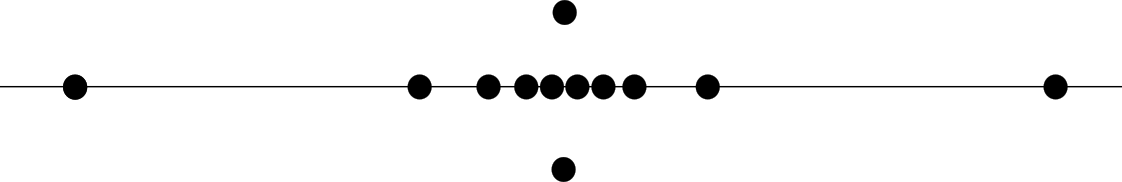} 
    \end{minipage}
    \caption{Left: Angles in arithmetic progression on a circle. Right: Angles in arithmetic progression on a line.}
    \label{f:examples}
\end{figure}

Other than the above two examples, we do not know of  point configurations with fewer than $\Omega(N^2)$ distinct angles. Can one carefully conjecture that the only way of getting fewer than $\Omega(N^{2-\epsilon})$ angles is by way of the above two examples? 

Currently, we do not know how to prove any nontrivial lower bound for the number of distinct angles even in the case of
 the points being {\em in general position} that is having {\em no three on a line and no four on a circle}.
As for the upper bound, one can get $O(N^2)$ distinct angles by taking $N$ points on a quadric. Namely, one can take a right polygon and scale it into an ellipse (leaving, if one wishes, only a quarter of the ellipse to ensure that the points are in general position). One can also consider  points, with abscissae in an arithmetic progression on the parabola $y=x^2$ or points, whose abscissae are in a geometric progression, on the hyperbola $y=1/x$. In each of these configurations, one easily sees that pairs of points determine $O(N)$ line slopes, hence altogether $O(N^2)$ angles between pairs of lines. A little more calculation would be required to show $\Omega(N^2)$ is the lower bound for the number of angles. We have verified this for the parabola and hyperbola.

If an $N$-point set determines $O(N)$ distinct directions, then it will determine $O(N^2)$ distinct angles.
Fleischmann, the first author, Miller, Palsson, Pesikoff and Wolf \cite{fleischmann2023distinct} (see the references therein describing an earlier construction by projecting a higher-dimensional cube) have recently pointed out that in a point set on the logarithmic spiral, one can have $O(N^2)$ angles without having $O(N)$ directions.
Elekes \cite{elekes1999linear} conjectured that a set determining $O(N)$ distinct directions must contain at least {\em six} (and possibly many more) points on a quadric, not necessarily irreducible, so one may say six on a quadric or three on a line.
This conjecture is wide open as are many structural questions concerning extremal point configurations in $\mathbb R^2$.
Is it related to the question we ask about the number of distinct angles? Six points on a quadric in the Elekes conjecture may turn out to be a fairly large proportion of the points, and in this case it would not be difficult to show that the quadric must be a circle or reducible. {\em Reducibility is meant over the reals.}
But it is generally not clear to what extent the open projective structure questions concerning point sets determining few directions (or more generally $\Omega(N^2)$ collinear point triples) relate to metric ones about the sets with few angles.
It seems to be fair to say that  the theme of counting angles in plane point sets has been somewhat neglected, versus, say counting distances.
However, we have not succeeded so far to prove a nontrivial lower bound other than under a fairly strong assumption made in this paper, which would not be satisfied if, say the point set were a grid $B\times B$. In the latter case, which proved to be elusive to us, Roche-Newton recently succeeded in establishing a nontrivial lower bound $\Omega\left(|B|^{2+\frac{1}{14}}\right)$, using a rather different approach \cite{olly_cp}.

Questions of giving bounds concerning the number of various geometric objects, arising in plane point configurations become easier if there is some order on the point set one take advantage of. To this effect, what we prove in this paper can be paralleled to Section 4 of the recent paper by Solymosi \cite{solymosi2023structure}, whose Theorem 15 claims that if an $N$-point set in {\em convex position} -- namely forming a vertex set of a convex polygon -- yields $O(N^{1+\epsilon})$ directions, then $\Omega(N^{1-\delta(\epsilon)})$ points must lie on a quadric. (The quantitative relation between the parameters $\epsilon$ and $\delta$ was not not spelt out explicitly but can be.) We use the convex position assumption in roughly this way: that 
the point set $P$ has two positive proportion subsets $P_1,\, P_2$, that is both of size $\Omega(N)$, such that each of the points from $P_2$ ``sees'' the points in $P_1$ in the same order. This can be seen in Figure \ref{fig:MutuallyAvoiding}. If one considered a line through any point of $P_2$, as the angle of the line with the horizontal increases, the points of $P_1$ are encountered in the same order. 

We will also use a second characterisation, namely the points of $P_1$ lie in the intersection of some fixed set of open half-planes, determined by all lines through pairs of distinct points of $P_2$. For each of these lines one chooses one of the two open hyperplanes it borders. See \cite[Lemma 19]{solymosi2023structure}.

We could have based a weaker qualitative version of our main result here on Solymosi's arguments but this would be unwieldy, costly and not immediate. Just showing that quadrics with many points should be lines or circles would require quantitative estimates on expanders of Elekes-R\'onyai type. On the other hand, our proof is not dissimilar, and the reader familiar with \cite{solymosi2023structure} will see this. What we call the order assumption is nearly synonymous to what some authors call {\em mutually avoiding sets}.

\clearpage
Let us formalise the {\em order assumption} we will use throughout.
\begin{definition}[Order assumption]\label{def:orderAssumption}
    Let $P$ be a finite point set in $\R^2$. Suppose, $P$ contains two disjoint sets $P_1,P_2$ with, respectively, $N_1$ and $N_2$ elements, and such that for every $x\in P_1$, the order of slopes of lines, connecting $x$ with the points $\{p_1,p_2,\ldots\}$ of $P_2$ is the same, up to a cyclic shift. In particular, we assume that no point of $P_1$ lies on a line connecting two or more points of $P_2$.
\end{definition}
\begin{center}
    \begin{figure}[ht]
        \centering
        \includegraphics[scale=0.7]{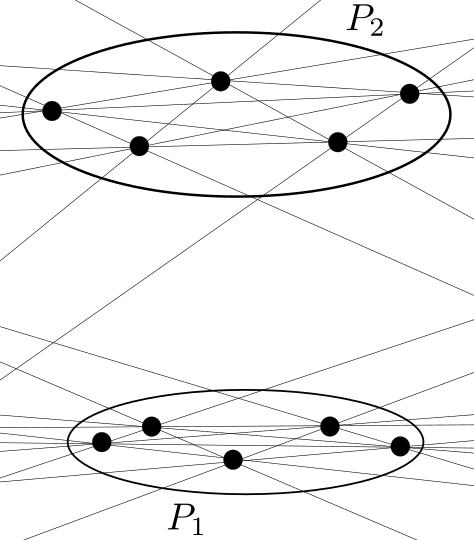}
        \caption{Point set  $P=P_1\cup P_2$ satisfies the \emph{order} assumption. The sets $P_1$ and $P_2$ are mutually avoiding.}
        \label{fig:MutuallyAvoiding}
    \end{figure}
\end{center}
\vspace{-0.7cm}
We can now formulate the main results in this paper.

\begin{theorem} \label{t:main}
Suppose, a $N$-point set $P\subset \R^2$, satisfies the order assumption with $|P_1|,\,|P_2|=\Omega(N)$, and $P$ determines $KN$ distinct angles. Then $K=\Omega(N^{1/4})$ or there are $\,\Omega(N/K)$ co-circular points in $P$.
\end{theorem}
From the above Theorem, we derive the main result as follows.
\begin{theorem}\label{t:unc} Suppose, a  $N$-point set $P\subset \R^2$, satisfies the order assumption with $|P_1|,\,|P_2|=\Omega(N)$, and $P$ is not contained in the union of points on a circle and its centre.  Then $P$ determines $\Omega\left(N^{\frac{5}{4}}\right)$ distinct angles.
\end{theorem}
If $P$ is a set of points in convex position, then one can choose $P_1$ and $P_2$, with both $N_1,N_2 \geq \lfloor N/2\rfloor - 1$. Indeed, any horizontal line can intersect at most two points. Slide a horizontal line up until it partitions the set into exactly equal pieces, call these pieces $P_1$ and $P_2$. We do not include points on our dividing line in either $P_1$ or $P_2$. (If $|P|$ is odd, for example, then there will be a single point on our dividing line which is not included. In the even case there maybe two points not included). We thus have the following Corollary.
\begin{corollary}\label{c:convexNotInCircle}
Suppose $P\subset \R^2$ is a  $N$-point set in convex position. If $P$ is not contained in a circle, then $P$ determines $\Omega\left(N^{\frac{5}{4}}\right)$ distinct angles.
\end{corollary}

\subsection{Discussion}
We finish the introduction with a discussion of the proof of Theorem \ref{t:main} and Theorem \ref{t:unc}. We then discuss the order assumption and generalisations to the above theorems.

To take advantage of the order assumption, we are guided by the following heuristic. Suppose $N$ points define $\Theta(N)$ distinct angles. Then if we take two pairs of random points $(p,q)$ and $(s,t)$ in $P_2\times P_2$, the expectation of the number of vertices $x \in P_1$, such that the angles $pxq$ and $sxt$ are equal, is  $O(1)$. Indeed, if we take two random angles at a typical vertex $x$ -- a Beck point, namely such that there are $\Omega(N)$ directions from $x$, the odds that the two angles are equal are $O(1/N)$. However,  since there are altogether $O(N)$ distinct angles with the vertex at $x$, the value of an angle $pxq$, where $p$ and $q$ are {\em neighbours} (or close to one another) in the natural ordering, would repeat $\Omega(N)$ times. 
If the order assumption is satisfied, then for every $x\in P_1$, the order of points in $P_2$, arising by looking at the directions of the lines connecting $x$ with the points of $P_2$ is the same. Since all $x$ in $P_1$ agree on what being neighbours relative to the ordering in $P_2$ means, this implies that if $(p,q)$ and $(s,t)$ are two random pairs of {\em neighbours} in $P_2$, the expected number of points $x\in P_1$, so that angles $pxq$ and $sxt$ are equal becomes $\Omega(N)$. 

The proof of our main Theorem \ref{t:main} relies on the quantitative version of this observation.  A similar heuristic underlies Solymosi's reasoning in \cite[Section 4]{solymosi2023structure}. We remark that using order has recently been fruitful in many instances. In this paper, as well as \cite[Section 4]{solymosi2023structure} it is indispensable for a non-trivial lower bound.  In other instances earlier non-trivial results had been obtained by means of incidence theory, and taking advantage of proximity in order would strengthen them. This was done in a recent paper by Solymosi and Zahl \cite{solymosi2024improved} apropos of the so-called Elekes-Szab\'o problem, asking for an upper bound on the number of zeroes of a trivariate real polynomial $f(x,y,z)$ of  a $O(1)$ degree, on a Cartesian product $(x,y,z)\in A\times B\times C$.

Additive structure is again used to prove Theorem \ref{t:unc}. The points of the regular $n$-gon form an arithmetic progression of angles. If we add a point at the center of the circle, the new angles also form an arithmetic progression.
Any other point in the plane cannot share this additive structure. See Figure \ref{fig:additiveCircle}. One proves Theorem \ref{t:unc} by playing the angles created between points on the circle with angles created by a point external to the circle, one of these must be large.
\begin{center}
    \begin{figure}[h!]
        \centering
        \includegraphics[scale=0.4]{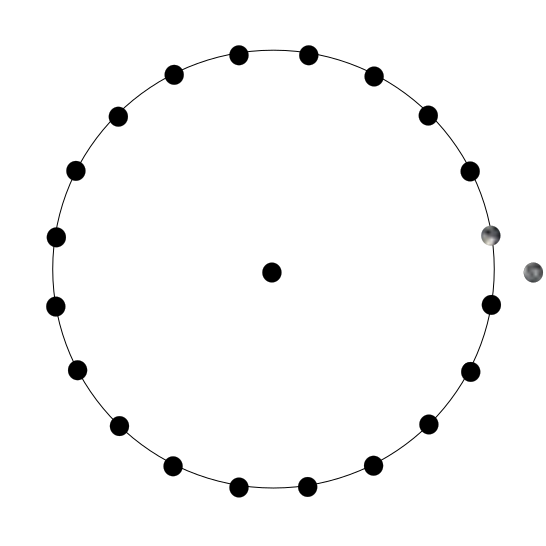}
        \caption{One plays the angles created at the shaded points off against each other. The convexity of the circle ensures that the angles created cannot share additive structure, so one must grow.}
        \label{fig:additiveCircle}
    \end{figure}
\end{center}
That a set and a convex image cannot both contain additive structure has been the focus of lots of recent work on the {\em convexity versus sumset} question. See \cite{bradshaw2022higher, bradshaw2023growth, hanson2022higher, hanson2023convexity, roche2023better, ruzsa2021distinct, ruzsa2022sumsets, stevens2022sum}, and references therein for more details. The circle provides the necessary convexity here and we quote the best known at the time of our writing this paper estimate due to Stevens and Warren \cite[Corollary 8]{stevens2022sum}, which was recently slightly improved by Bloom \cite{bloom_cont}. However, to prove Theorem \ref{t:main} (where several ensuing incidence bounds will be dominated by bound \eqref{e:upper}) the weaker exponent $\frac{5}{4}$ in the next bound, from the fountainhead paper by Elekes, Nathanson and Ruzsa \cite{ENR} already suffices.
\begin{theorem}\label{t:StevenWarren}
For any sufficiently large real finite set $B$ and a strictly convex function $f$, one has
\begin{equation}\label{e:css}
|B-B|+|f(B)-f(B)| \gg |B|^{13/10-\epsilon}\,.
\end{equation}
\label{th:conv}
\end{theorem}

\medskip

\paragraph{\textbf{Generalisations}}
The reader can easily generalise the forthcoming proof to an asymmetric case of the parameter values $(N, N_1, N_2)$.  We present the case $N_1,N_2=\Omega(N)$, whose asymmetric adaptation is straightforward after chasing through the proof's estimates. We only remark on the still nontrivial case $N^\epsilon\leq N_1\leq N_2$ and $N_2 = \Omega(N)$. 

In our argument, the latter assumption may be slightly weakened. We do not pursue this, as the random point set will not satisfy the order assumption in any meaningful way.
\begin{remark} $\,$
\begin{enumerate}
\item Suppose the $N$-point set satisfies the order assumption with $|P_1|=N_1$ and $|P_2|=\Omega(N)$. Then $K\gg N_1^{1/4}$ or $\,\Omega(N/K)$ points of $P$ are co-circular. 
 \item The proof of Theorem \ref{t:main} holds under a weaker assumption that there are at most $KN$ angles with a base in $P_2$ from every point in $P_1$.\end{enumerate}
\end{remark}

In a general set, one would not expect $N_1$ and $N_2$ this large. For $N$ non-collinear points in $\R^2$, Aronov et al. \cite{aronov1991crossing} and Valtr \cite{valtr1997mutually} show the largest mutually avoiding sets $P_1$, $P_2$ are both of size $\Theta(N^{1/2})$. One can thus take $N_1, N_2=\Omega(N^{1/2})$. For other results using avoiding sets see Mirzaei and Suk \cite{mirzaei2020positive} and Pack, Rubin and Tardos \cite{pach2019planar}.
Convex position is even harder to achieve. For $N$ non-collinear points one can only guarantee $N_1=N_2=\Omega(\log N)$. See \cite{suk2017erdHos} and references therein on the Erd\H{o}s-Szekeres convex polygon problem.

The Elekes-Szab\'o problem is closely related to the already mentioned Elekes-R\'onyai problem (see, e.g. the references in \cite{solymosi2024improved}), asking for the lower bound on cardinality of the range of a bivariate bounded degree polynomial $g(x,y)$ on a Cartesian product $(x,y)\in A\times B$. The problem has been extended to rational functions $g(x,y)$, see \cite{bukh2012sum, de2018survey, elekes2000combinatorial}.

A trivial bound cannot be improved, of course, if after changing each variable independently, one gets $g(x,y) \to h(x+y)$, with the new variables in an Abelian group. This group may be $(\mathbb R,+)$, $(\mathbb R^*,\times)$ or $S^1$, the circle.
The latter case arises when $g(x,y)$ is a scalar function of $\frac{x+y}{1-xy}$, after the variable change $x=\tan \phi$, $y=\tan \psi$, so $\frac{x+y}{1-xy} = \tan(\phi+\psi)$. This establishes an immediate connection between the Elekes-R\'onyai problem with estimating the minimum number of angles with a fixed vertex for a point set, lying on a bounded degree curve.
\section{Proof of Theorem \ref{t:main}}

The proof is organised in several steps, with the goal being to double count the number of edges of a bipartite graph. We first provide formal definitions used throughout the proof, including the bipartite graph. In the second subsection, we prove a lower bound on the number of edges using pigeonholing and convexity. We finish the proof by considering this bipartite graph as an incidence graph, and provide an upper bound by examining the geometry of the associated curves.

\subsection{Definitions}~

We will assume throughout this section that $P$ is a set if $N$ points in $\R^2$ that satisfies our order assumption with $P\supseteq P_1\sqcup P_2$, where $|P_1|, |P_2| = \Omega(N)$.
To distinguish between points in the sets $P_1$ and $P_2$ we will use $x$ and $y$ to denote points in $P_1$ and $p, q, s, t$ to denote points in $P_2$.

A key object in the proof is the set of directions formed at a point. Let $x\in P_1$, we will use ${D}_x(P_2)$ to denote the set of directions $x$ forms with points in $P_2$.

It will be useful for us to represent these directions as the angles they make with a fixed reference line $l$, say the horizontal axis. We need to be precise about the specific orientation of the angle that we are taking. To this end, orientate the horizontal axis as travelling from negative to positive, orient a line through $x$ and a point $p\in P_2$ as travelling from $x$ to $p$. We specify that we want the anti-clockwise angle formed by rotating the oriented horizontal axis onto the oriented line $\overrightarrow{xp}$ about their point of intersection. See Figure \ref{fig:directionAngle}. One can see that if $p$ is below $x$ according to the vertical axis, the direction will be in $(\pi, 2\pi)$, otherwise it will be in the range $(0,\pi)$.
\begin{figure}[ht]
    \centering
    \includegraphics[scale=0.3]{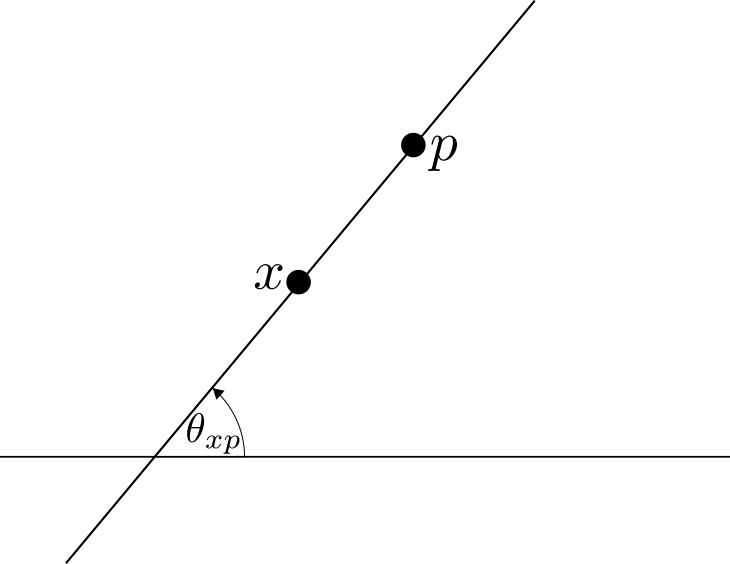}
    \caption{A direction $\overrightarrow{xp}$ represented as the oriented angle $\theta_{xp}$.}
    \label{fig:directionAngle}
\end{figure}

Thus, by deciding the orientation of the horizontal axis we can ensure that, for any given point $x$ at least half of the orientated angles centred at $x$ lie in the range $(0,\pi)$. So, throughout, we consider oriented angles in the range $(0,\pi)$ only.

We assume that $P$ has at most $KN$ distinct oriented angles, with $K=\Omega(1)$. We will use $A(P)$ to denote the set of distinct oriented angles formed by triples of points in $P$.
\[ A(P) = \{ \angle pqr: p,q,r \in P\} \cap (0, \pi).\]

For the upper bound to work, we need to take some care setting up our bipartite graph $G$. Let $G$ have vertex set $P_1 \cup (P_2\times P_2)$. We will use $|G|$ to denote the number of edges of $G$ and use $(x, (p,s)) \in G$ to denote that $(x,(p,s))$ is an edge in G.

We will need to make frequent use of the order assumption induces on $P_2$. It is not technically an order, but we will only use the concept of neighbours (and near-neighbours), so the below definition suffices.
\begin{definition}\label{def:OrderOnP2}
For each point $x$ in $P_1$, draw small circles about each such point so that no point of $P_2$ is contained within any such circle. Consider the projection of $P_2$ onto each such circle.
By the order assumption, if we start at a fixed point $p$ in $P_2$ and go around these circles clockwise, then we will meet all of the points of $P_2$ in the same order.

We thus define $t$ to be the \textit{neighbour} of $s$ if $t$ is the first element of $P_2$ one finds when one travels clockwise from $s$ on any such circle.
Similarly, $t$ \textit{is within 5 of} $s$ if $t$ is one of the first 5 elements of $P_2$ one meets when travelling clockwise from $s$ on any of these circles.
\end{definition}

To define the edges of the graph, we use this order. Let $p,s \in P_2$ and let $q$ be the neighbour of $p$ and $t$ the neighbour of $s$. Then
\[ (x,(p,s)) \in G \Leftrightarrow \angle pxq = \angle sxt \text{ and both angles are in } (0,\pi). \]
Unfortunately, we will have to adapt this definition for a technical reason. 
In our argument, $p$ and $q$ being neighbours is too strict, we need instead that $q$ is within 5 positions of $p$ in the order. 
This is going to increase the number of edges in $G$, but only by an amount that effects the implicit constants in our upper bound. 
We will also need that $q=t$ ensures that $p=s$. This turns out to be a simple consequence of the order assumption.
We state this in the following lemma, it is proved in Section \ref{sec:proofOfLemmata}.
\begin{restatable}{lemma}{neighbourAssumption}\label{l:neighbourAssumption}
    Suppose, a $N$-point set $P\subset \R^2$, satisfies the order assumption with $|P_1|,\,|P_2|=\Omega(N)$. 
    Let $x\in P_1$ and $p,q,s \in P_2$. If the orientated angles $\angle pxq$ and $\angle sxq$ are equal then $p=s$.
\end{restatable}
It is the order assumption, that all $x\in P_1$ agree as to whether or not $q$ is a neighbour of $p$ in $P_2$, that has enabled us to identify the second set of vertices in $G$ with $P_2\times P_2$, rather than a larger subset of $(P_2\times P_2)\times (P_2\times P_2)$. Our argument fails if the second set of vertices set is as large as $P_2^4$.

\subsection{Lower bound: Pigeonholing}
We aim to prove the following proposition.
\begin{proposition}\label{prop:lowerBound}
Suppose, a  $N$-point set $P\subset \R^2$, satisfies the order assumption with $|P_1|,\,|P_2|=\Omega(N)$. Suppose that $P$ defines at most $KN$ distinct angles. Let $x\in P_1$ and $p,q,s,t \in P_2$, by the order assumption we have an order $<$ on $P_2$.

Define a bipartite graph $G$ on the vertex set $P_1 \cup (P_2\times P_2)$ where $(x,(p,s))$ is an edge if there is some $q$ within 5 of $p$ with respect to the $<$ order and, similarly, there is some $t$ within 5 of $s$ so that $\angle pxq = \angle sxt$. Then,
\[|G| \gg N^3/K^2.\]
\end{proposition}
Under this looser definition of the graph, for a fixed $p$ and $s$ there are multiple choices for $q$ and $t$ which can give this equality of angles.
As long as there is one such pair, among the 25 possibilities, we add the edge $(x,(p,s))$ to $G$.

Take any $x \in P_1$, with ${D}_x(P_2)$ being the set of directions from $x$ to the points of $P_2$. We can choose the zero direction, depending on $x$, so that there are at least $|P_2|/2$ directions in the interval $(0,\pi)$. Let $D_x$ be the set of these directions, written in the increasing order. Note that $|D_x| \geq |P_2|/2 \gg N$ as the angles each element of $P_2$ makes with the zero direction will be unique (and we ensured that there are at least $|P_2|/2$ such angles in $(0, \pi)$).

Consider positive differences in $D_x-D_x$, as one can see from Figure \ref{fig:diffsAsAngles}, they are a subset of distinct angles formed by $P$ at $x$. 
\begin{figure}[ht]
    \centering
    \includegraphics[scale=0.35]{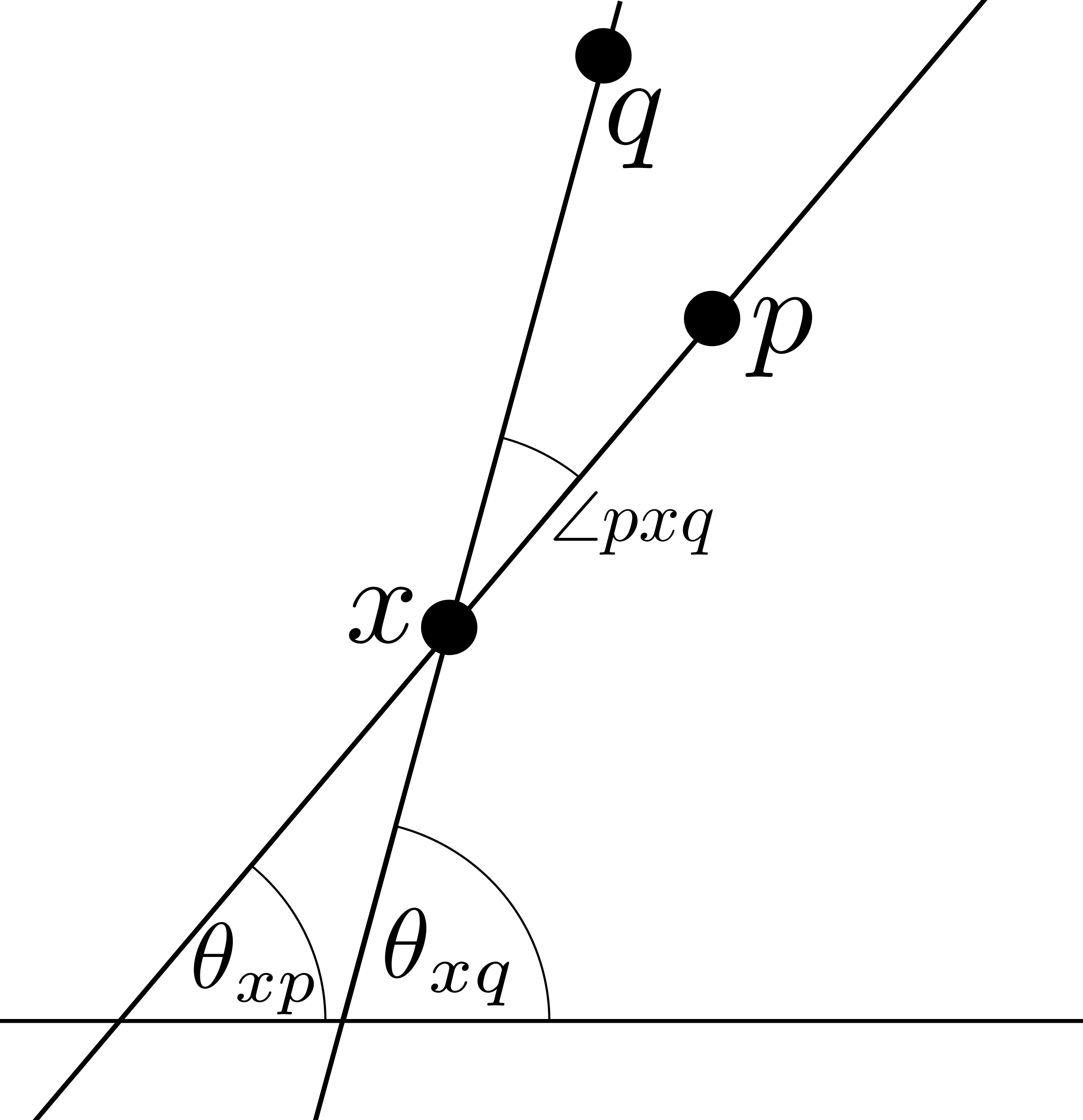}
    \caption{The angle $\angle pxq$ is the difference $\theta_{xq}-\theta_{xp}$, where $\theta_{xp}, \theta_{xq} \in D_x$.}
    \label{fig:diffsAsAngles}
\end{figure}

Noting that any negative difference in $D_x -D_x$ is $-1$ times a positive difference, and hence also an angle at $x$. For example, in Figure \ref{fig:diffsAsAngles} one can see the negative difference $\theta_{xp}-\theta_{xq}$ as the negative angle $\angle qxp$. As $A(P)$ contains only the positive angles, we have the bound
\begin{equation}\label{one}
|D_x-D_x| \leq 2|A(P)| \leq 2KN.
\end{equation}

We want to convert the above bound concerning the difference set $D_x-D_x$ into one about the set $D_x + (D_x - D_x)$. This will then allow us to use that consecutive intervals in a real set $D$ are quite dense with elements of $D+(D-D)$. This observation goes back to at least the paper \cite{ruzsa2021distinct} by Solymosi. Recently it found use in several works on convexity and sumsets of reals, see, for example, \cite{bradshaw2022higher, bradshaw2023growth, hanson2022higher, hanson2023convexity,roche2023better, ruzsa2022sumsets}. 
To achieve this we need a corollary of Pl\"unnecke-Ruzsa, see e.g. \cite[Theorem 1.2]{gyarmati2008plunnecke} and \cite[Corollary 1.5]{katz2008slight}, \cite[Exercise 6.5.1]{tao2006additive}.
\begin{corollary}\label{col:ThinPluennecker}
Let $X, B_1, \ldots, B_k$ be finite subsets in a commutative group. Suppose that $|X+B_i| \leq \alpha_i |X|$. Then there is $X' \subseteq X$ with $|X'| > |X|/2$ so that 
\[|X' + B_1 + \cdots + B_k| \leq \alpha_1\cdots\alpha_k 2^k |X|\]
\end{corollary}

Applying Corollary \ref{col:ThinPluennecker} directly gives us a $D_x' \subseteq D_x$ where $|D_x'| \geq |D_x|/2$ and
\begin{equation}\label{eq:Pluennecker}
    |D'_x+(D_x-D_x)| = |-D_x + D_x' + D_x| \leq 4\cdot\frac{|-D_x+D_x'||-D_x+D_x|}{|-D_x|}
\end{equation}
We note that $|-D_x+D_x'| \leq |D_x - D_x|$ and apply \eqref{one} twice to obtain
\begin{equation}\label{two}
|D'_x+(D_x-D_x)| \leq 16 K^2 N.
\end{equation}
Since each $D_x'$ is ordered, consider the neighbouring intervals $[d_i,d_{i+1})$ with endpoints in $D_x'$. Call such an interval {\em normal} if 
\begin{enumerate}[(i)]
    \item The number of elements of $D'_x+(D_x-D_x)$ in this interval is less than $49 K^2$.
    \item The interval $[d_i,d_{i+1})$ contains at most $4$ elements of $D_x$ (not including the endpoints).
\end{enumerate}
Note that the constant 49 is not optimised as we will suppress it in the asymptotic notation. We write it explicitly here to streamline the proof. We need two facts about normal intervals, we include a quick proofs of both at the end of the section.
\begin{claim}\label{claim:NormalIntervals}~
    \begin{enumerate}
        \item At least a third of the intervals $[d_i,d_{i+1})$ with endpoints in $D'_x$ are normal.
        \item The diameter of normal intervals is among the first $49 K^2$ positive differences in $D_x$.\\
        Formally, write the positive elements of $D_x-D_x$ in the increasing order as $\delta_1 < \delta_2 < \cdots$. Then, for all normal intervals $[d_i,d_{i+1})$, we have
        \[ d_{i+1}-d_i < \delta_{49 K^2 + 1}. \]
    \end{enumerate}
\end{claim}

Hence, for each $x\in P_1$ we take $D''_x$ to be the set of left endpoints $d_i$ in $D_x'$ so that the neighbouring interval $[d_i,d_{i+1})$ is normal. Property (1) of Claim \ref{claim:NormalIntervals} gives $|D''_x|=\Omega(N)$. By property (2) of Claim \ref{claim:NormalIntervals} we can partition $D_x''$ into $\Omega(N/K^2)$ classes by the difference $d_{i+1}-d_i$.

Let $\nu_x(\delta)$ be the size of the difference class where the endpoints of a normal interval $[d_i, d_{i+1})$, with $d_i \in D_x''$ and $d_{i+1}\in D_x'$ satisfy $d_{i+1}- d_i = \delta$. That is
\[\nu_x(\delta) = \{d_{i} \in D_i'' : d_{i+1} \in D_x' \text{ and } d_{i+1}- d_i = \delta\}\]
Let $N(x)$ denote the neighbours of $x$ in our bipartite graph $G$. 
By the definition of edges in our bipartite graph, any pair of normal intervals at $x$ chosen from the same difference class gives a distinct edge in $G$.

Indeed, property (ii) of normality ensures there are at most 4 elements of $D_x$ in the interval. Thus,
the point of $P_2$ associated with $d_{i+1}$ is within 5 positions of the point associated with $d_{i}$, using Definition \ref{def:OrderOnP2}. 
That the two normal intervals have the same difference, $\delta = d_{i+1}-d_i$, ensures that the angle is the same. 
The edges are distinct as the $d_i$ are distinct, thus the points associated to the $d_i$ in $P_2$ are distinct too.

Thus,
\[|N(x)| \geq \sum_{j=1}^{49 K^2} \nu^2_x(\delta_j), \]
where the support of the sum is using property (2) of Claim \ref{claim:NormalIntervals}. Thus, using Cauchy-Schwarz, we have that
\begin{equation*}
    |N(x)| \geq \frac{1}{49 K^2} \left(\sum_j \nu_x(\delta_j) \right)^2 = \frac{|D_x''|^2}{49 K^2} \gg \frac{N^2}{K^2}.
\end{equation*}
The final inequality using the earlier estimate $|D_x''| = \Omega(N)$. Summing over all $x \in P_1$ and using the assumption that $|P_1| = \Omega(N)$ gives 
\begin{equation}\label{e:gr}
    |G| \gg \frac{N^3}{K^2}.
\end{equation}
It remains to prove Claim \ref{claim:NormalIntervals}.

\begin{proof}[Proof of Claim \ref{claim:NormalIntervals}]
Let $D_x = {d_1<d_2<\cdots<d_m}$, where the $d_i$ are in $(0,\pi)$ and $m =\Omega(N)$ (as guaranteed by Corollary \ref{col:ThinPluennecker}). 
We will first show that claim (1) is true.

Suppose that property (i) of normality fails for more than two thirds of the intervals $[d_i, d_{i+1})$. Each failing interval thus contains more than $CK^2$ elements of $D_x' + (D_x - D_x)$.
Using \eqref{two}, we have
\[
    \sum_{i=1}^{m} |(D_x' + (D_x - D_x)) \cap [d_i, d_{i+1})| \leq |D_x' + (D_x - D_x)| \leq 16 K^2 N.
\]
We also have the following lower bound 
\[\sum_{i=1}^{m} |(D_x' + (D_x - D_x)) \cap [d_i, d_{i+1})| \geq \frac{2|D_x'|}{3} \cdot CK^2 \geq \frac{|D_x|}{2} \cdot \frac{2}{3} \cdot C K^2 \geq \frac{C}{3} \cdot K^2 N.\]
If we take $C=(16\times 3) + 1 = 49$ we have a contradiction. So, two thirds of our intervals have property (i).

Similarly, assume that two thirds of our intervals $[d_i, d_{i+1})$ are not normal as they do not satisfy property (ii). Then 
\[
    \sum_{i=1}^{m} |D_x \cap [d_i, d_{i+1})| \leq |D_x|.
\]
We also have the following lower bound 
\[\sum_{i=1}^{m} |D_x \cap [d_i, d_{i+1})| \geq \frac{2}{3}\cdot|D_x'|\cdot 4 \geq \frac{2}{3} \cdot \frac{D_x}{2} \cdot 4 = \frac{4}{3}|D_x|,\]
which provides the sought contradiction.

As two thirds of our set satisfy properties (i) and (ii) independently, at least a third of the set must satisfy properties (i) and (ii) simultaneously. This is exactly what part (1) of the claim stated.

For property (2), note that $d_{i+1} - d_i$ is a difference of angles in $(0, \pi)$ and thus we can order the positive differences in $D_x - D_x$ using the real order. Denote the first $49 K^2 + 1$ of these by $\delta_1 < \delta_2 < \cdots < \delta_{49 K_4} <  \delta_{49 K_4 + 1}$. Let $[d_i, d_{i+1})$ be a normal interval with $\delta = d_{i+1} - d_i$. Then, it suffices to prove that $\delta < \delta_{49 K^2 + 1}$.

Suppose otherwise, that $\delta \geq \delta_{49 K^2 + 1}$. Then for all $1 \leq j \leq 49 K^2 + 1$
\[d_i + \delta_j \leq d_i + \delta = d_{i+1}.\]
So the interval $[d_i, d_{i+1})$ contains $49 K^2 + 1$ sums of the form $d_i + \delta_j$. Recalling that the $\delta_j$ are in $D_x - D_x$ shows that the normal interval contains strictly more than $49 K^2$ elements of $D_x + (D_x-D_x)$. This is a contradiction of property (i) of a normal interval.
\end{proof}

\subsection{Upper bound: Representing $|G|$ via incidences of curves and points}~

With the lower bound for the number of edges established in \eqref{e:gr} we move on to the upper bound. We prove the following proposition.
\begin{proposition}\label{prop:IncidenceBound}
Suppose, a  $N$-point set $P\subset \R^2$, satisfies the order assumption with $|P_1|,\,|P_2|=\Omega(N)$. Suppose that $P$ defines at most $KN$ distinct angles.

Let $G$ be the bipartite graph as described in Proposition \ref{prop:lowerBound}. Let $M$ be the maximum number of points of $P$ on a circle. Then at least one of the following holds
\begin{enumerate}
    \item $|G| = O\left( N^{5/2}\right).$
    \item $|G| = O\left( NM^2\right).$
\end{enumerate}
\end{proposition}
The first part of this section is dedicated to creating and analysing a set of curves $\Gamma$, so that incidences between curves in $\Gamma$ and points in $P_1$ count the edges in $G$. Our curves can be reducible and have multiplicities, so we take care to account for this. We conclude the section by proving an incidence bound which accounts for the multiplicities of the curves.

\subsection*{Defining the Curves}~

Consider a family of at most $N^2$ real curves $\gamma_{pqst}$ which are defined by the condition that the cotangents of the angles $pxq$, where $q$ is the neighbour of $p$ in $P_2$ and $sxt$, where $t$ is the neighbour of $s$ in $P_2$ are equal. Since the angles considered are in $(0, \pi)$ this is in one-to-one correspondence with the equality of angles.

We will prove the following lemma about the curves $\gamma_{pqst}$. The proof is technical and thus delayed until Section \ref{sec:proofOfLemmata}.
\begin{restatable}{lemma}{PolyDescription}\label{l:polyDescription}
Let $p,q,s,t$ be such that $p\neq s$, $t\neq q$, $p\neq t$ and $s\neq q$. Let $c_i$ be real numbers that depend only on $p,q,s$ and $t$.
There is a non-zero polynomial
\begin{equation}\label{e:curves}
   f_{pqst}(x_1,x_2) = c_7(x_1^3+x_1x_2^2) + c_6(x_2^3+x_2x_1^2) + c_5x_1^2 + c_4x_2^2 + c_3x_1x_2 + c_2x_1 + c_1x_2 + c_0, 
\end{equation}
such that for all $x\in \R^2$, if $\angle pxq = \angle sxt$, then $x \in Z(f_{pqst})$. Furthermore, $c_7=c_6=0$ if, and only if, $p-q=s-t$.
\end{restatable}
As discussed below, we will take
\begin{equation}\label{e:curvesAsVarities}
    \gamma_{pqst} = Z(f_{pqst}).
\end{equation}
Lemma \ref{l:polyDescription} allows us to link edges in our graph to points in $P_1$ being incidence to curves $\gamma_{pqst}$. Indeed, let $p$ and $s$ be in $P_2$ with $q$ and $t$ the neighbours of $p$ and $s$ in $P_2$.
By definition of our graph $G$ we have an edge $(x, (p,s))\in G$ if, and only, if $\angle pxq = \angle sxt$,  
So, by Lemma \ref{l:polyDescription}, we have that
\[ (x,(p,s)) \in G \Rightarrow x \in \gamma_{pqst}. \]
Unfortunately, the curves $\gamma_{pqst}$ may have multiplicities.
To see this recall that $\mathbb{RP}^9$ is the space of bivariate real polynomial of degree at most three.
By Lemma \ref{l:polyDescription}, our curves $\gamma_{pqst}$ lie in a seven-dimensional subspace of $\mathbb{RP}^9$.
Since the family $\Gamma$ is defined by $p,q,s,t\in \R^2$ it is defined by eight real parameters, so the curves $\gamma_{pqst}$ may have multiplicities. Furthermore, if some of the polynomials happen to be reducible, they may share irreducible components. The reducibility and the multiplicity are handled in the forthcoming Lemma \ref{l:reducible} and Lemma \ref{l:two}.

To handle the multiplicities, let $\Gamma=\Gamma(P_2)$ be the multi-set of curves $\gamma_{pqst}$ described above. For a curve $\gamma\in \Gamma$, let $m_\gamma$ denote its multiplicity in $\Gamma$.
We define the weighted incidence count $|I(P,\Gamma)|$ to be
\[|I(P,\Gamma)| = \sum_{p\in P}\sum_{\gamma\in\Gamma}m_\gamma \delta_{p\in \gamma},\]
where $\delta_{p\in \gamma}$ is one if $p\in \gamma$ and zero otherwise.
Lemma \ref{l:polyDescription} thus gives us the following corollary.
\begin{corollary}\label{cor:GraphBoundByIncidence}
    Let $G$ be the graph as defined in Proposition \ref{prop:lowerBound} and let $\Gamma(P_2)$ be the set of curves as defined in \eqref{e:curvesAsVarities}. Then,
    \[ |G|\ll |I(P_1,\Gamma(P_2))|.\]
\end{corollary}
\subsection*{Geometry of Curves}~

We may assume, that $p\neq s$, $t\neq q$, $p\neq t$ and $s\neq q$, slightly restricting the above graph $G$, still calling it $G$. One can check that this does not affect the lower bound \eqref{e:gr} on the number of edges.

A curve $\gamma_{pqst}$ arises as follows. Since the angles considered are in $(0,\pi)$, any $x\in\gamma_{pqst}$ must lie in the intersection $Q$ of positive half-planes based on the lines $pq$, $st$ (that is, the two normals form plus ninety degrees angles with the vectors $q-p$ and $t-s$). See Figure \ref{f:PointsOnGamma}.
\begin{figure}[ht]
    \centering
    \includegraphics[scale=0.4]{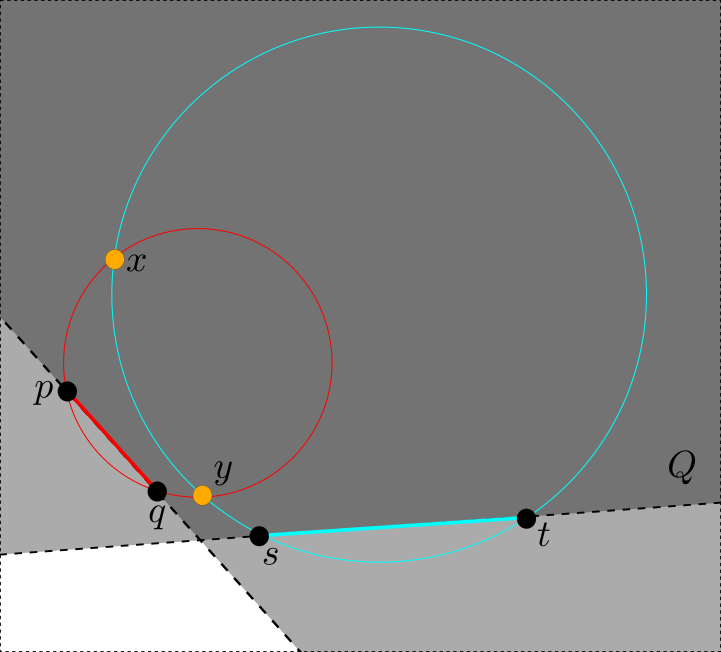}
    \caption{The points $x$ and $y$ lie on the intersection of the circles, in the the intersection of the positive half-plane $Q$ and thus on the curve $\gamma_{pqst}$.}
    \label{f:PointsOnGamma}
\end{figure}

Let us cover the positive half-plane, based on the line $pq$ by a pencil of circles, containing $pq$ as a chord. By the \emph{inscribed angle theorem}, all points in the intersection of such a circle and the same half-plane as the centre form the same angle. See Figure \ref{f:inscribedCircle}.
\begin{figure}[ht]
    \centering
    \includegraphics[scale=0.4]{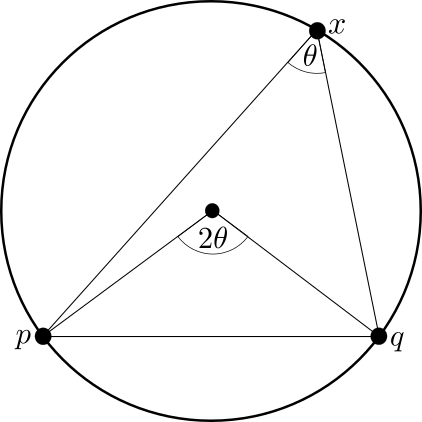}
    \caption{Any point $x$ on the circle in the positive half-plane (as drawn) forms an angle of $\theta$ when the circle has radius $R=\frac{|p-q|}{2\sin\theta}$.}
    \label{f:inscribedCircle}
\end{figure}

Each of these circles have radii $R\geq |p-q|/2$. Let us regard circles whose centres are in the positive half-plane as having positive signed radius $R$ and those centred in the negative half-plane as having negative signed radius $-R$. Now, fix $p,s$ in $P$ with $q$ and $t$ their neighbours respectively. For a given $\theta$, find $R$ so that the circle of radius $R$ through $p$ and $q$ inscribes all points which form an angle of $\theta$. There is then a $\lambda$ so that the circle of radius $\lambda R$ through $s$ and $t$ inscribes all points that from the same angle $\theta$ with $s$ and $t$. 

For a fixed angle $\theta$ the circle which represents points forming an angle $\theta$ with $p$ and $q$ has radius $R=|p-q|/2\sin\theta$. Using a similar formula for the circle through $s$ and $t$, we see that $\lambda = \frac{|s-t|}{|p-q|}$. Two such intersection points are shown in Figure \ref{f:PointsOnGamma}.

The curve $\gamma_{pqst}$ is restricted to $Q$, where it arises by intersecting, for each signed $R$ (equivalently, the angle $\theta$), the circle of radius $R$ from a pencil with the chord $pq$ with the circle from a pencil with the chord $st$ and the signed radius $\lambda R$. In particular, unless the two circles have the same radius (which requires $|p-q|=|s-t|$), their intersection occurs in at most two points, as well as if $|p-q|=|s-t|$ but the four points $p,q,s,t$ are not co-circular. The only special case is when the points $p,q,s,t$ are co-circular and the chords $pq$ and $st$ have the same length and orientation along the circle of some radius, which is uniquely defined by these $p,q,s,t$. We call this value of $R$ {\em special}, for this value of $R$ (or the corresponding inscribed angle $\theta$) alone yields the whole irreducible component of the curve $\gamma_{pqst}$. The second irreducible component is easily seen to be the symmetry axis between the points $p,t$ (and $q,s$), see the forthcoming lemmata for details.



We will replace henceforth each curve $\gamma_{pqst}$ by the whole zero set of the polynomial $f_{pqst}(x)$ in Lemma \ref{l:polyDescription}. This will imply counting the equality of oriented angles $pxq$ and $sxt$ mod $\pi$, including the zero angle.  In terms of the above geometric description, viewing $\gamma_{pqst}$ as the zero set of a polynomial means no longer confining the intersections of the pairs of circles from the two pencils, parameterised by the signed value of $R$, to the set $Q$, formed by the intersection of the two positive half-planes.

Lemma \ref{l:polyDescription} tells us that the curve $\gamma_{pqst}$ will have degree lower than three only if $p-q=s-t$. In this case one cannot have a special value of $R$ parameterising the curve $\gamma_{pqst}$, because if the parallelogram $pqts$ can be inscribed in a circle of some radius $R$, it has to be a rectangle. But then the signed values of the parameter, corresponding to the circle pencils on the chords $pq$ and $st$ will be $R$ and $-R$, owing to orientation. For an easy calculation one may fix, say $p=(1,0)$, $q=(-1,0)$ and $s=(1,c)$, $t=(-1,c)$, with some $c$. Intersecting pairs of circles of equal radii, whose centres lie on the ordinate axis and are separated by $c$ and lie both above or below the two horizontal line segments $pq, \,st$ shows that $\gamma_{pqst}$ is a hyperbola, degenerating in two mutually perpendicular lines $pt$ and $qs$ when $c=2$, that is $pqts$ is a square.

Thus, the existence of the special value of $R$ makes the cubic polynomial, whose zero set is the curve $\gamma_{pqst}$, reducible over the reals. Indeed, if $R$ is the special value of the signed parameter, the circle  with the unsigned radius $|R|$, containing the points $p,q,s,t$ is one component of $\gamma_{pqst}$, the other component being the straight line  through the centre of the circle, relative to which the segments $pq$ and $ts$ are symmetric. We call this case {\em Scenario R1}.

Is there another geometric scenario when $\gamma_{pqst}$ is described by a reducible polynomial?  It is easy to see that if the points $q,s$ are symmetric relative to the line, connecting the points $p,t$, then the latter line belongs to the zero set of  $\gamma_{pqst}$, which is therefore reducible, without the existence of a special value of $R$. We call this case {\em Scenario R2} and will show in the forthcoming lemma that {\em Scenarios R1 and R2} are the only ones that occur apropos of reducibility. Moreover, under {\em Scenario R2}, the second component of $\gamma_{pqst}$ is a circle, with a centre on the line $pt$ or a line perpendicular to the latter line if  $pt$, $qs$ are diagonals of a rhombus, in which case $\gamma_{pqst}$ has degree two, rather than three.

\begin{restatable}{lemma}{reducible}\label{l:reducible} Suppose that $P$ satisfies the order assumption, with $p,q,s,t \in P_2$.
Then, the curve $\gamma_{pqst}$ is reducible under the following two scenarios only.

R1: The segments $pq$ and $ts$ are symmetric relative to some line. Furthermore,  $\gamma_{pqst}$ is the union of this line and the circle containing the four points $p,q,s,t$ and corresponding to a special value of the angle, unless also $p-q=s-t$. In the latter case, $\gamma_{pqst}$ is quadratic and reducible as the union of two mutually perpendicular lines. See Figure \ref{f:R1}.

R2: Points $q,s$ (or $p,t$) are symmetric relative to the line $pt$ (respectively $qs$). The latter line yields one irreducible component of  $\gamma_{pqst}$. The other component is either a circle centred on the above line or the line $qs$ (respectively $pt$)  if $pt$, $qs$ are diagonals of a rhombus. In both cases the points $q,s$ belong to the latter component. See Figure \ref{f:R2}.
\end{restatable}

\begin{figure}[ht]
    \centering
    \includegraphics[scale=0.4]{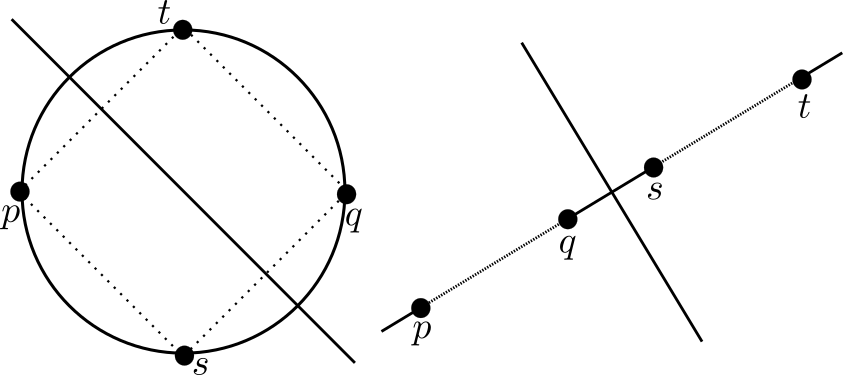}
    \caption{The two R1 cases for $\gamma_{pqst}$.}
    \label{f:R1}
\end{figure}
\begin{figure}[ht]
    \centering
    \includegraphics[scale=0.4]{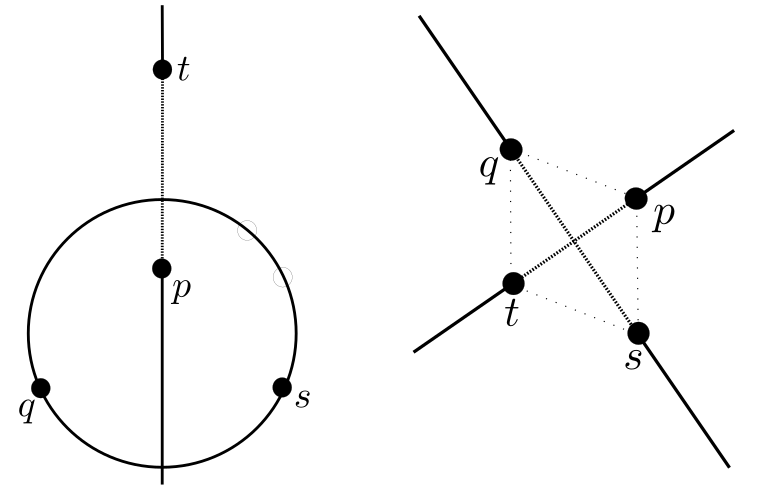}
    \caption{The two R2 cases for $\gamma_{pqst}$.}
    \label{f:R2}
\end{figure}
We present the proof of Lemma \ref{l:reducible} in a separate section along with the proof of the next lemma, which bounds multiplicity of the curves $\gamma_{pqst}$.
Indeed, there are altogether eight real parameters, coming from
$\{(p,q), (s,t)\}$, but the family of polynomial coefficients arising after rearranging equation \eqref{e:curves}  is easily seen to lie in a seven-dimensional projective space, as there are only two linearly independent coefficients over the four third degree monomials. Thus generically a one-dimensional family in the $(p,q), (s,t)$ space would yield the same curve $\gamma_{pqst}$.

\begin{restatable}{lemma}{multiplicity}\label{l:two}
For each pair of neighbours $(p,q)\in P_2\times P_2$, there are at most two pairs $(s,t)$ determining the same full curve  $\gamma_{pqst}$ (that is not just a component of in the reducible Scenario R1) or a circular component of $\gamma_{pqst}$ under Scenario R2.
 \end{restatable}
 
The conclusion of Lemma \ref{l:two} is that the multiplicity of the full curve $\gamma_{pqst}$ cannot exceed $2N$. However, under Scenario R1, a circular component may belong to more curves. Indeed, suppose we have $M$ evenly spaced points of $P_2$ on a circle. Then the circle itself is an irreducible component of $\Omega(M^2)$ curves $\gamma_{pqst}$, where $(p,q),(s,t)$ are any two neighbouring pairs of points on the circle. 

\subsection*{Proof of Proposition \ref{prop:IncidenceBound}}~

Assuming Lemmata \ref{l:reducible} and \ref{l:two} we can, with the help of some additional results from incidence theory, conclude the proof of Theorem \ref{t:main}.
The first result concerns weighted incidences, we prove it in Section \ref{sec:proofOfLemmata}.
\begin{restatable}{lemma}{incidenceBound}\label{l:incidence}
Let $P$ be a finite set of points and $\Gamma$ a  finite multi-set set of irreducible curves where any two distinct curves in $\Gamma$ intersect at most $C$ times. For each $\gamma \in \Gamma$ let $m_\gamma \in \N$ represent the multiplicity of $\gamma$ in $\Gamma$. Let $m_{\max}$ be the largest $m_\gamma$ and let $M\in \R$ be such that
\[\sum_{\gamma \in \Gamma}m_\gamma \leq M.\]
Define the weighted number of incidences as $$|I(P,\Gamma)| := \sum_{\gamma\in \Gamma} \sum_{p\in P}m_\gamma \delta_{p\in\gamma}.$$ Then
\[|I(P,\Gamma)| \leq m_{\max}|P| + C^{1/2}M|P|^{1/2}.\]
\end{restatable}
The second result is a standard bound one can find many proofs of, for example the proofs of \cite[Lemma 7.3, Lemma 10.16]{guth2016polynomial} which are stated for points but apply equally to lines.
\begin{lemma}\label{l:veryRichLines}
    Let $P \subseteq \R^n$ be a set of $N$ points. Let $\Gamma$ be a finite set of curves in $\R^n$ so that no two distinct curves in $\Gamma$ intersect more than twice. Let $\Gamma_k$ be the set of curves in $\Gamma$ which contain at least $k$ points of $P$. Then, if $k\geq 4N^{1/2}$ we have that
    \[|L_k| < \frac{2N}{k}.\]
\end{lemma}

We decompose our set of incidences into three cases. Incidences with degree three irreducible curves $\gamma_{pqst}$;  incidences with the circular components of $\gamma_{pqst}$; and incidences with linear components of $\gamma_{pqst}$. We use $|I(P_1,\Gamma_{\text{irr}})|, |I(P_1,\Gamma_{\text{circ}})|$ and $|I(P_1,\Gamma_{\text{lin}})|$ to represent the incidences in each of these cases respectively. Then Lemma \ref{l:reducible} gives us that
\begin{equation}\label{e:3cases}
    |I(P,\Gamma)| = |I(P,\Gamma_{\text{irr}})| + |I(P, \Gamma_{\text{circ}})| + |I(P, \Gamma_{\text{lin}})|.
\end{equation}
Using Corollary \ref{cor:GraphBoundByIncidence}, to prove Proposition \ref{prop:IncidenceBound}, it suffices to establish good bounds on each of the three weighted incidence counts above.
We will deal with each case separately.

\paragraph{\textbf{Incidences with Irreducible Curves}}
Suppose the first term of \eqref{e:3cases} dominates, that is that most of the geometric incidences with the point set $P_1$, corresponding to the edges of the graph $G$ come from irreducible curves $\gamma_{pqst}$. Then, since the multiplicity of each curve is $2N$, and as two distinct curves intersect at at most $9$ points, we can use Lemma \ref{l:incidence} to give us that
\begin{equation}\label{e:upper}
|I(P,\Gamma_{\text{irr}})| \ll N^2 +N^2\sqrt{N}.
\end{equation}
Thus, we have the first conclusion of Proposition \ref{prop:IncidenceBound}.

\paragraph{\textbf{Incidences With Circular Components.}} We deal with circles differently depending on which case of Lemma \ref{l:reducible} the circle comes from. 
We refer to the number of incidences from scenario R1 using $|I(P_1,\Gamma_{\text{circ}}^{R1})|$,
similarly $|I(P_1,\Gamma_{\text{circ}}^{R2})|$ gives the number of incidences with circles in scenario R2.

\paragraph{\textbf{Circles whose centre lies on a bisecting line.}} These are the circles from scenario R2 of Lemma \ref{l:reducible}. 
Lemma \ref{l:two} tells us that such circles have multiplicity at most $2N$. Thus, using that two circles intersect in at most 2 places, 
Lemma \ref{l:incidence} then gives us that
\[ |I(P_1,\Gamma_{\text{circ}}^{R2})| \ll N^2 +N^2\sqrt{N}.\]
Thus, we have the first conclusion of Proposition \ref{prop:IncidenceBound}.

\paragraph{\textbf{Circles which contain all four points $p,q,s$ and $t$.}} These are the circles from scenario R1 of Lemma \ref{l:reducible}. 
Once we have decided $p$ and $s$ on the circle, by our order assumption, we know $q$ and $t$. 
Thus, if we fix a circle, finding the number of ways to choose such a pair $(p,s) \in P_2^2$ on this circle will give a bound on how many $\gamma_{pqst}$ the circle could belong to. 
We now break down our circles by the number of points of $P_2$ they contain and use this to establish the bound on the weight of a fixed circle.

As there are at most $O(N^2)$ curves $\gamma_{pqst}$, the total weighted number of circles must also be $O(N^2)$. Thus, circles with at most $8N^{1/2}$ points of $P_1$ on them can contribute $O(N^{5/2})$ weighted incidence. This gives us the first conclusion of Proposition \ref{prop:IncidenceBound}.

We thus assume that all our circles contain strictly more than $8N^{1/2}$ points of $P_1$. We are now in the regime of Lemma \ref{l:veryRichLines}. The standard argument down to \eqref{eq:fat} makes rigorous the heuristics that if $M\gg\sqrt{N}$ is the maximum number of points of $P$ on a circle, then the maximum possible number of weighted multiplicities will be achieved when $P$ is evenly distributed between $N/M$ circles, when each circle may contribute up to $\sim M^3$ incidences.

Let the number of points $k$ on the circle be such that $k > 8\sqrt{N}$. At this stage we need to take some care with our point sets. Indeed, fix a circle which arises as an irreducible component in scenario R1, we want to use the number of points of $P_2$ on the circle to bound
its multiplicity and the number of points of $P_1$ on the circle to control how this multiplicity adds to our weighted incidence sum.

Define $\mathcal{C}_{=k}(P)$ to be the set of circles that arise from reducibility scenario R1 and contain exactly $k$ points in $P$,
with $\mathcal{C}_{\geq k}(P)$ defined similarly but for at least $k$ points of $P$.
Any circle $C\in \mathcal{C}_{=k}(P)$ can contain at most $k$ points in $P_2$, and therefore its multiplicity as to the Scenario R1, namely the number of pairs $(p,s)$, so that both $pq$ and $st$ are chords of $C$ (all chords having the same length), is at most $k^2$. Recall that $p$ defines $q$ and $s$ defines $t$. Clearly, we also have  $|P_1\cap C|\leq k$.

Thus, if $M$ is the maximum number of points of $P$ on a circle, we see that since 
\[
|I(P_1,\Gamma_{\text{circ}}^{R1})| \ll \sum_{k=\sqrt{N}}^M k^2 \cdot k\cdot |\mathcal{C}_{=k}(P)|.
\]
Using that $|\mathcal{C}_{=k}(P)| = |\mathcal{C}_{\geq k}(P)| -|\mathcal{C}_{\geq (k+1)}(P)|$ gives us that
\[
|I(P_1,\Gamma_{\text{circ}}^{R1})| \ll \sum_{k=\sqrt{N}}^M k^2 \cdot |\mathcal{C}_{\geq k}(P)|
\]
Using Lemma \ref{l:veryRichLines} gives us that
\begin{equation}\label{eq:fat}
|I(P_1,\Gamma_{\text{circ}}^{R1})| \ll \sum_{k=\sqrt{N}}^M k \cdot N \ll NM^2.
\end{equation}
This is the second conclusion of Proposition \ref{prop:IncidenceBound}.

\paragraph{\textbf{Incidences With Linear Components}}
Finally, we consider, and dismiss, incidences owing to straight lines. Note that Lemma \ref{l:veryRichLines} applies equally to lines as to circles. Thus, by the same argument as above, incidences supported on lines with fewer than $\sqrt{N}$ points give at most $O(N^{5/2})$ incidences.

Let $L$ the maximum number of points (of $P_1$) on a straight line and let $\Gamma_{\text{lin}, k}$ denote the linear components which intersect at least $k$ points of $P_1$. We thus estimate $|I(P_1,\Gamma_{\text{lin}})|$ using
\begin{equation}
    |I(P_1,\Gamma_{\text{lin}})| \ll N^{5/2} + \sum_{k=\sqrt{N}}^L |I(P_1, \Gamma_{\text{lin}, k})|.
\end{equation}

To estimate the latter sum we need to count the weight such a line can contribute to $|I(P_1,\Gamma)|$.
By Lemma \ref{l:reducible}, any line that forms a component of a curve in $\Gamma$ must be the perpendicular bisector of a pair $(p,t)\in P_2^2$. To deal with the these bisectors we introduce {\em bisector energy}.

Let $l$ be some line in $\Gamma_{\text{lin}, k}$ define $n(l)$ to be the number of line segments in $P_2$ where $l$ is the perpendicular bisector. If we define $B(p,t)$ to be the perpendicular bisector of $pt$ then we can define $n(l)$ as
\[ n(l) = \left|\{(p,t) \in P_2^2 : B(p,t) = l\}\right|.\]
The quantity $\sum_{l} n^2(l)$, with the summation extended to all lines, has been referred in literature as {\em bisector energy} of $P_2$. This was studied in \cite{lund2016bisector, lund2020bisectors} as well as \cite{murphy2022pinned} interpreting bisector energy via point-plane estimates. We bound this using the following theorem \cite[Theorem 3.1]{lund2020bisectors}. We adapt the statement to our setting, although for our purposes, with all the incidence estimates being dominated by \eqref{e:upper}, a trivial bound $n(l)\leq N$ would already do the job.

\begin{theorem}[Lund--Petridis (2020)]\label{thm:LundPetridis}
Let $P \subseteq \R^2$ be a finite set of size $|P| = N$. Suppose that any line or circle contains at most $M'$ points of $P$. If
\[Q_{M'} = |{(p, q, s, t) \in P^4 : B(p, t) = B(q, s)}|,\]
then
\[Q_{M'} \ll M' N^2 + \log^{1/2}(N) N^{5/2}.\]
\end{theorem}
Thus, such a line $l$ may contribute to $|I(P_1,\Gamma_{\text{lin}})|$ the quantity $n(l)k$.
Summing over lines in $\Gamma_{\text{lin}, k}$ and using Cauchy-Schwarz, we see that
\begin{equation*}
|I(P_1, \Gamma_{\text{lin}, k})| =
\sum_{l \in \Gamma_{\text{lin}, k}}n(l)k
\leq k\left(\frac{N}{k} \sum_{l\in \Gamma_{\text{lin}, k}} n^2(l)\right)^{1/2}.
\end{equation*}
The quantity $\sum_{l} n^2(l),$ with the summation, extended to all lines, is exactly the {\em bisector energy} of $P_2$, called $Q_{M'}$ in Theorem \ref{thm:LundPetridis}. We thus have the bound
\[
\sum_{l} n^2(l) \ll N^{5/2}\log^{1/2} N + M'N^2\,
\]
where $M'$ is the maximum number of points (of $P_2$) on a circle or line. Combining with the above, we have that
\begin{equation}\label{eq:krichLinesEstimate}
|I(P_1, \Gamma_{\text{lin}, k})| \ll \left(N^{7/2}\log^{1/2} N + M'N^3\right)^{1/2} \cdot k^{1/2}.
\end{equation}

We split up the second sum into dyadic intervals and apply \eqref{eq:krichLinesEstimate} to obtain
\begin{align*}
    \sum_{k=\sqrt{N}}^L |I(P_1, \Gamma_{\text{lin}, k})| &\ll \left(N^{7/2}\log^{1/2} N + M'N^{3/2}\right)^{1/2}\sum_{i=(1/2)\log N}^{\log L}2^{i/2}\\
    &\ll L^{1/2} N^{7/4} \log^{1/4} N + \sqrt{LM'} N^{3/2}.
\end{align*}

As, $L\leq N$ and $M'\leq N$, both terms here are $O(N^{5/2})$. This gives us conclusion (1) of Proposition \ref{prop:IncidenceBound}.

This concludes all cases and so we have established Proposition \ref{prop:IncidenceBound}.\hspace{4.55cm} $\Box$

\subsection*{Conclusion of the proof of Theorem \ref{t:main}}~

To finish the proof of Theorem \ref{t:main} we bring together Proposition \ref{prop:lowerBound} and Proposition \ref{prop:IncidenceBound}. We start by assuming that the first conclusion of Proposition \ref{prop:IncidenceBound} holds i.e. $|G|\ll N^{5/2}$. We then use the lower bound proved from Proposition \ref{prop:lowerBound}, giving us that
\[ \frac{N^3}{K^2} \ll |G| \ll N^{5/2},\]
and thus
\[ N^{1/4} \ll K.\]
If the second conclusion holds then, with $M$ the maximum number of points of $P$ on a circle, we have that $|G| \ll NM^2$.
We then use the lower bound proved from Proposition \ref{prop:lowerBound}, giving us that
\[ \frac{N^3}{K^2} \ll |G| \ll NM^2,\]
and thus
\[\frac{N}{K} \ll M.\]
This concludes the proof of Theorem \ref{t:main}.
\hspace{82mm} $\Box$

\section{Proof of Theorem \ref{t:unc}} The proof has two main steps. The first step is to show that if all of $P$, {\em but for one point} were to lie on the union of a circle and its centre, $P$ would necessarily yield $\Omega(N^{13/10-\epsilon})$ distinct angles. This is done by reducing the issue to an application  of Theorem \ref{th:conv}.

The second step is a combinatorial argument that analysing the side of the alternative, presented by Theorem \ref{t:main}, when $P$ has many points on a circle, we may effectively assume that this circle contains a positive proportion of $P$, for otherwise one can refine the incidence graph $G$ underlying the proof of Theorem  \ref{t:main} and get a better estimate.

\subsection{Invoking Theorem \ref{th:conv}}~
Consider some fixed point $p_a= (\cos a,\sin a)$ for $t=a$, where $t$ is the variable on the unit circle. Let $(x,0)$ be a point on the abscissa axis, with $x\geq 0$. Let $P$ be the union of some set $A$ of points on the circle with $\{x\}$.
Our goal is to establish that in order to justify the estimate of the corollary it suffices to consider the union of all angles with the vertices at $x$ or $a$ only. Observe that, say if $a_1,a_2\in A$, then the angle $a_1xa_2$ is the difference of the angles formed by the slopes of the lines, connecting $a_1$ and $a_2$ with $x$.

Without loss of generality we may assume that all points of $A$ lie in the positive quadrant.

Let
$$
\alpha = \arctan\left(\frac{\sin t-\sin a}{\cos t - \cos a}\right)
$$
be the angle of the slope of the line, connecting points $p_t$ and $p_a$ on the unit circle at the angles $t$ and $a$. From elementary geometry, looking at the isosceles triangle $p_a0p_t$ (where $0$ is the origin) $d\alpha/dt=\frac{1}{2}$.

Now consider the angle of the slope of the line, connecting $p_t$ and $(x,0)$.

This angle is
$$
\beta =  \arctan\left( \frac{\sin t}{\cos t-x}\right)\,.
$$

\begin{center}
    \begin{figure}[ht]
        \centering
        \includegraphics[scale=0.5]{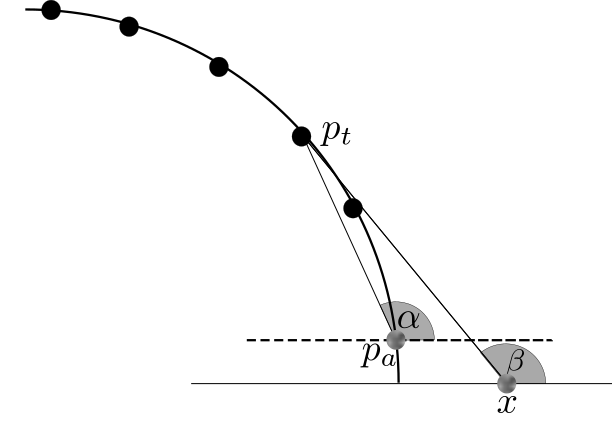}
        \caption{Angle $\alpha$ is a strictly convex function of $\beta$. So, the number of angles at either $p_a$ or $x$ must grow.}
        \label{f:PointsNotAllOnACircle}
    \end{figure}
\end{center}
 
Let $t=t(\beta)$,  implicit differentiation yields
$$
1 = \frac{1-x\cos t}{x^2-2x\cos t + 1}\frac{dt}{d\beta}.
$$
Thus,
$$
\frac{d\alpha}{d\beta} = \frac{1}{2} \frac{x^2-2x\cos t + 1}{1-x\cos t} =  1 +  \frac{x^2-1}{2(1-x\cos t)}\,.
$$
Differentiating once more
$$
\frac{d^2\alpha}{d\beta^2} = - \frac{1}{2} x(x^2-1) \frac{(x^2-2x\cos t + 1)\sin t} {(1-x\cos t)^3}\,.
$$
Since we assume that $t\in (0,\pi/2)$, for nonnegative $x\neq 0,1$,  the second derivative
$\frac{d^2\alpha}{d\beta^2} $ is sign-definite. 

We can now apply Theorem \ref{th:conv}, with $B$ being the set of slopes $\beta$ of the lines, connecting a point in the set $A$ on the circle with $(0,x)$ and $f(B)$ the function $\alpha(\beta)$.
It follows that if $A$ is now a point set in any circle, $a\in A$ and $x$ is any point in the plane which is neither on the circle nor its centre, then the union, over $a_1,a_2\in A$, of the sets of all angles $a_1aa_2$ and $a_1xa_2$ has cardinality $\Omega\left( |A|^{13/10-\epsilon}\right)$. We repeat that for our purposes we can use the weaker bound $\Omega\left( |B|^{5/4}\right)$ by Elekes et al. \cite{ENR}. Although this is irrelevant for the proof of our main result in view of assumption on $P$ (namely that both sets $P_1$ and $P_2$ have cardinality $\Omega(N)$), we remark that the latter estimate applies, after a similar calculation, to the number of angles with the bases in some set $A$ on a straight line and vertices at two points $x_1,x_2$, which are not symmetric relative to the line. For generalizations of local convexity estimates, see a recent paper by Mansfield \cite{sam_conv}.

\label{sub}

\subsection{Conclusion of the proof of Theorem \ref{t:unc}}~
We start with an upshot of the forthcoming argument. In the order assumption one may assume without loss of generality that the sets $P_1$ and $P_2$ are disjoint (by worsening the absolute constants hidden in the $O,\Omega$ symbols). Let us call a circle {\em rich} if it supports more than $N^{3/4}$ points of $P_2$; the number of rich cirlcles  by the standard Cauchy-Schwarz argument of Lemma \ref{l:veryRichLines}, is $O(N^{1/4})$.

Recall the upper bound in the case of R1 reducibility scenario from  
the proof of Theorem \ref{t:main}, with $M$ being the maximum number of points of $P$ on a circle:
\begin{equation}\label{eq:bad} |I(P_1, \Gamma_{\text{circ}}^{\text{R1}})| \ll \sum_{C\in \mathcal C} |P_2\cap C|^2|P_1\cap C|\leq NM^2,\end{equation}
where $\mathcal C$ is a collection of circles each supporting $\gg\sqrt{N}$ points of $P$. This was {\em the only} case, which could possibly dominate the $O(N^{5/2})$ bound on the number of edges in the graph $G$, the latter bound in particular controlling the reducibility scerario R2.

Let $P_2'$ be the part of $P_2$ supported on the union of rich circles. The claim of Theorem \ref{t:unc} follows from  
the proof of Theorem \ref{t:main} immediately  if either $P_2\setminus P_2' =\Omega(N)$ (by redefining $P_2$ as $P_2\setminus P_2'$) or if $\Omega(N)$ elements of $P_1$ are supported outside the union of rich circles (by redefinig $P_1$ as the set of the latter $\Omega(N)$ elements).

We therefore assume that $P_2=P_2'$ and that all of of $P_1$ lies on the union of rich circles. Observe that
we are also done if there is a circle with, say $\gg N^{.99}$ points of $P_2$ after an application of Theorem \ref{th:conv} as described in the previous subsection.
Then the following lemma, whose proof is given in the next section, enables one to close the argument.
\begin{restatable}{lemma}{Pups}\label{lem:denseOnCircles}
Suppose that we have the union of two disjoint sets $P_1$ and $P_2$ in $\mathbb{R}^2$, supported on a finite set of circles $\mathcal{C}$ in $\mathbb{R}^2$. Then one of the two statements is true:
\begin{itemize}
    \item There are disjoint subsets $\mathcal{C}_1$ and $\mathcal{C}_2$ of $\mathcal{C}$ so that $\mathcal{C}_i$ supports a positive proportion of $P_i$.
    \item There is a  circle $C\in\mathcal{C}$ that supports a positive proportion of both $P_1$ and $P_2$.
\end{itemize}
\end{restatable}
If the second conclusion of Lemma \ref{lem:denseOnCircles} holds, Theorem \ref{t:unc} follows by the conclusion of Section \ref{sub}.
If the first conclusion of Lemma \ref{lem:denseOnCircles} holds, note that since $|\mathcal C|\ll N^{1/4}$, once can throw away from both sets $P_1$ and $P_2$, of cardinality $\Omega(N)$, points lying at the union of pairwise intersections of circles from $\mathcal{C}_1$ and $\mathcal{C}_2$, whereupon Theorem \ref{t:unc} follows after redefining $P_i$ as $P_i\cap \mathcal C_i$ and the observations earlier.

 $\Box$

\section{Proof of Lemmata}\label{sec:proofOfLemmata}
Throughout this section we will need the wedge product. For two vectors $a=(a_1,a_2)$ and $b = (b_1,b_2)$ both in $\R^2$ we define this as
\begin{equation*}
    a \wedge b = a_1b_2 - a_2b_1.
\end{equation*}
Let $\theta$ be the angle between $a$ and $b$, we will need the following properties of the wedge product.
\begin{itemize}
    \item $a \wedge b = -b \wedge a$.
    \item $a \wedge (b \wedge c) = (a \wedge b) \wedge c$.
    \item $a \wedge a = 0$.
    \item $|a \wedge b| = |a||b|\sin(\theta)$.
\end{itemize}
We recall Lemma \ref{l:neighbourAssumption}.
\neighbourAssumption*
\begin{proof}
    Assume that $p\neq s$. As the orientated angles $\angle pxq = \angle sxq$, we must have that the distinct points $p$ and $s$ lie on the same half-line with end-point $x$.
    This contradicts our order assumption as $x$ is a point of $P_1$ that lies on a line containing the two points $p$ and $s$ in $P_2$.
\end{proof}

Recall Lemma \ref{l:polyDescription}.
\PolyDescription*
\begin{proof}[Proof of Lemma \ref{l:polyDescription}]~

Suppose that $p,q,s$ and $t$ are any points in $P_2$ (at this stage they need not be close in the sense of Definition \ref{def:OrderOnP2}). We want to find the locus of points $x$ so that $\angle pxq = \angle sxt$. If these angles do agree then so must their cotangents. Treating points in $P$ as vectors in $\R^2$ and using the dot-product formula give us
\begin{equation}\label{eq:dotProd}
   \cos(\angle pxq) = \frac{(p-x)\cdot(q-x)}{|p-x||q-x|}.  
\end{equation}
Using the final property of the wedge product stated above, we have that
\begin{equation}\label{eq:wedgeProd}
    \sin(\angle pxq) = \frac{|(p-x)\wedge(q-x)|}{|p-x||q-x|} = \frac{(p-x)\wedge(q-x)}{|p-x||q-x|}.
\end{equation}
The final equality uses that our angles are in $(0,\pi)$ and thus sine is strictly positive. We can thus divide \eqref{eq:dotProd} by \eqref{eq:wedgeProd} to give us
\begin{equation}
    \cot(\angle pxq) = \frac{(p-x)\cdot(q-x)}{(p-x)\wedge(q-x)}.
\end{equation}
Setting $\cot(\angle pxq) = \cot(\angle sxt)$ and using the above rules of the wedge product to expand the denominator, where we assume that $(p,q)\neq (s,t)$, gives the equation
\begin{equation}\label{e:wedgecurves}
\frac{|x|^2 - x\cdot (p+q) + p\cdot  q}{p\wedge q - (p-q)\wedge x} = \frac{|x|^2 - x\cdot (s+t) + s\cdot  t}{s\wedge t - (s-t)\wedge x}.
\end{equation}
Fix $p,q,s,t$ and let $x=(x_1,x_2)$. We want to understand \eqref{e:curves} as a real polynomial with the variables $x_1$ and $x_2$. 

If rearrange equation \eqref{e:wedgecurves} by multiplying up the denominators it is easier to see that one has a degree 3 curve. The left-hand side of the rearranged formula \eqref{e:wedgecurves} becomes
\begin{equation*}
    (s\wedge t)(x_1^2 + x_2^2 - x_1(p+q)_1 - x_2(p+q)_2 + p\cdot q) + (x_1(s-t)_2 - x_2(s-t)_1)(x_1^2 + x_2^2 - x_1(p+q)_1 - x_2(p+q)_2 + p\cdot q)\,
\end{equation*}
where, say  $(p+q)_1$ is used to denote the first component of $p+q$. On the ohter side of the equality we have the same expression with the ordered pairs $(p,q)$ and $(s,t)$  swapped.

In particular, the left-hand side is
\begin{align*}
    &{\color{white}+~} (x_1^3+x_1x_2^2)(s-t)_2 - (x_2^3+x_2x_1^2)(s-t)_1 \\
    &+ x_1^2\left(s\wedge t - (s-t)_2(p+q)_1\right)\\
    &+ x_2^2\left(s\wedge t + (s-t)_1(p+q)_2\right)\\
    &+ x_1x_2\left( (p+q)_1(s-t)_1 - (p+q)_2(s-t)_2\right)\\
    &+ x_1\left( (s-t)_2(p\cdot q) - (s\wedge t)(p+q)_1 \right)\\
    &+ x_2\left( -(s-t)_1(p\cdot q) - (s\wedge t)(p+q)_2 \right)\\
    &+ 1 \cdot(p \cdot q)(s\wedge t).
\end{align*}
Thus, if we rearrange \eqref{e:wedgecurves} by multiplying up the denominators and moving all terms to the right-hand side we are left with a degree three polynomial in $x=(x_1,x_2)$. Call this resulting polynomial $f_{pqst}$. By construction, any point $x$ where $\angle pxq =\angle sxt$ will ensure that the two sides of \eqref{e:wedgecurves} are equal and thus $f_{pqst}(x)=0.$

Inspecting the coefficients of $x_1^3$ and $x_2^3$ above, one can see that $c_7=c_6=0$ if, and only if $p-q=s-t$.
\end{proof}

Recall Lemma \ref{l:reducible}.
\reducible*

\begin{proof}[Proof of Lemma \ref{l:reducible}]
If the curve $\gamma_{pqst}$ is reducible, as it had degree at most 3 it must contain a line. 

Without loss of generality, the line is $x_2=0$, and denoting $x_1=z$ we have that for all $z\in \R$,
\begin{equation}
\frac{z^2 - z(p_1+q_1) + p\cdot q }{z^2 - z (s_1+t_1) + s\cdot t} = 
\frac{z(p_2-q_2) + p\wedge q}{z(s_2-t_2) + s\wedge t}\,.
\label{e:pols}
\end{equation}
We assume that $p,q,s,t$ are all distinct.

The proof proceeds by rearranging the above as a cubic polynomial in $z$, which must be identically zero. This implies, comparing the coefficients at $z^3$ that
\begin{equation}s_2-t_2 = p_2-q_2:=a\,.\label{eq:a}\end{equation}
Suppose, $a=0$.  Then in \eqref{e:pols} either we have a constant $\frac{ p\wedge q}{ s\wedge t}$ in the right-hand side or $p\wedge q = s\wedge t =0$. In the former case, looking at $z^2$ terms, it follows that $\frac{ p\wedge q}{ s\wedge t}=1$. In both cases we can jump to the forthcoming conclusion \eqref{e:eq} and proceed thence to Scenario R1.

Assuming $a\neq 0$, we
proceed with coefficients at lower powers $z$ in the rearrangement of \eqref{e:pols}. This yields
\begin{equation}\label{e:coeff}\begin{aligned}
s\wedge t - p\wedge q &= a[  (p_1+q_1) - (s_1+t_1)],\\
a(p\cdot q - s\cdot t) &= (p_1+q_1)s\wedge t - (s_1+t_1)p\wedge q,\\
(p\wedge q)(s\cdot t) &= (s\wedge t)(p\cdot q) .
\end{aligned}
\end{equation}
Eliminate  $s\wedge t$  and $s\cdot t$ from the first two equations:
$$
\begin{aligned}
s\wedge t  & = p\wedge q + a[  p_1+q_1 - (s_1+t_1)],\\
 s\cdot t & = p\cdot q -a^{-1}[(p_1+q_1)( p\wedge q + a[  p_1+q_1 - (s_1+t_1)]) - (s_1+t_1)p\wedge q] \\
 &= p\cdot q - [  p_1+q_1 - (s_1+t_1)](p_1+q_1+a^{-1}p\wedge q) \,.
 \end{aligned}
$$
Substituting into the third equation yields
\begin{equation}\label{eq:prod}
[  p_1+q_1 - (s_1+t_1)] \cdot [ (p\wedge q) (p_1+q_1+a^{-1}p\wedge q)  + a (p\cdot q)] \;=\;0\,.
\end{equation}
If the first factor in square brackets is zero, it means 
$p_1+q_1= s_1+t_1$.
and further, from the first two equations of  \eqref{e:coeff}, $s\wedge t=p\wedge q$ and  $s\cdot t=p\cdot q$. The same, as we mentioned, is true when $a=p_2-q_2=s_2-t_2=0$.

To conclude the description of this scenario, summing up, we have
\begin{equation}\label{e:eq}
p_2-q_2=s_2-t_2,\qquad p_1+q_1 = s_1+t_1,\qquad p\wedge q = s\wedge t,\qquad p\cdot q = s\cdot t\,.
\end{equation}

Changing variables to $p+q=m,\;s+t=m'$, $p-q=l,\;s-t=l'$ transforms it to
$$
m_1=m_1',\;\;l_2=l_2', \;\; m_2l_1 = m_2'l_1',\;\; m_2^2-l_1^2 = m_2'^2 - l_1'^2\,.
$$ The solutions of the last two equations are either $m_2=m_2'$ and $l_1=l_1'$, to be dismissed as this means $(s,t)=(p,q)$, or $m_2=-m_2'$ and $l_1=-l_1'$.

This means, the segments $pq$ and $ts$ are symmetric relative to the $x_1$-axis. Thus, unless  both segments are vertical, the points $p,q,s,t$ are co-circular. Otherwise, as then $p-q=s-t$, the curve $\gamma_{pqst}$ is quadratic and reducible further as the union of two mutually perpendicular lines. This is the content of Scenario R1.

\medskip 
The alternative scenario arises when  the second term in square brackets in \eqref{eq:prod} is zero. Recalling that $a=p_2-q_2$ and $p\wedge q = p_1q_2-q_1p_2$, we have that
\begin{align*}
    (p\wedge q)(p_1+q_1+a^{-1}(p\wedge q)) + a(p\cdot q) &= 0\,\\
    (p\wedge q)(ap_1 + aq_1 + (p\wedge q)) + a^2(p\cdot q) &= 0\,\\
    (p\wedge q)(p_1p_2-q_1q_2) + (p_2-q_2)^2(p\cdot q) &= 0.
\end{align*}
After opening brackets and rearranging this becomes
\begin{equation}\label{e:deg}
p_2q_2[(p_1-q_1)^2 + (p_2-q_2)^2]=0\,.
\end{equation}
Recall that $p\neq q$, and the case $p_2=q_2=0$ has already been considered.

Thus, it remains to consider the case when \eqref{e:deg} holds with only one of $p_2,q_2$ being zero. Observe, that from symmetry we must have condition \eqref{e:deg} satisfied by $(s,t)$ replacing $(p,q)$ as well, hence with only one of $s_2,t_2$ being zero.

Then it suffices to proceed under assumption that  $p_2=0$ and $a=-q_2\neq 0$. The case $q_2=0$, rather than $p_2$, follows from symmetry of \eqref{e:pols}, relative to transposing simultaneously  within the pairs $(p,q)$ and $(s,t)$. 

Thus assume that that  $p_2=0$, $a=-q_2\neq 0$ and either (i) $s_2=0$ and $q_2=t_2\neq 0$, or (ii) $t_2=0$ and $s_2=-q_2\neq 0$. 
In the former case (i) continuing from \eqref{e:pols} gives us that
\begin{align*}
    \frac{z^2 - z(p_1+q_1) + p\cdot q }{z^2 - z (s_1+t_1) + s\cdot t} &= 
\frac{z(p_2-q_2) + p\wedge q}{z(s_2-t_2) + s\wedge t}\,\\
    \frac{z^2 + z(p_1+q_1) + p_1q_1}{z^2 + z(s_1+t_1) + s_1t_1} &= \frac{az-ap_1}{az-as_1}\,\\
    \frac{(z-p_1)(z-q_1)}{(z-s_1)(z-t_1)} &= \frac{z-p_1}{z - s_1}.
\end{align*}
This means $q=t$ and therefore, as $P$ satisfies the order assumption, Lemma \ref{l:neighbourAssumption} tells us $s=p$. Hence, this case gets dismissed.
In the latter case (ii) it follows from \eqref{e:pols} that
$$\frac{(z-p_1)(z-q_1)}{(z-s_1)(z-t_1)} = \frac{z-p_1}{z - t_1}\,,$$
Thus, the points $(p,t)$ are on the abscissa axis, $q_1=s_1$ and since $s_2=-q_2$, the points $q,s$ are symmetric with respect to the abscissa axis. This is Scenario R2.

At this point also observe that if $pqts$ it a rhombus, then Scenario R2 takes place simultaneously if we swap $(p,t)$ and $(s,q)$ around, so the other diagonal of the rhombus is also a component of the curve $\gamma_{pqst}$. Moreover, then $p-q=s-t$, so by \eqref{e:wedgecurves} the degree of the curve $\gamma_{pqst}$ is at most two, so there are no other components.

It remains to show that otherwise in Scenario R2, the second component of $\gamma_{pqst}$ is a circle. This is a calculation, once we return to \eqref{e:wedgecurves}, where we can set $t=0$, $p=(p_1,0)$, $q=(q_1,q_2)$, $s=(p_1,-q_2)$. Substituting in \eqref{e:wedgecurves} and rearranging yields
$$
x_2 [ (p_1-2q_1)(x_1^2+x_2^2) + 2(q_1^2+q_2^2)x_1 - p_1(q_1^2+q_2^2) ] =0,
$$
and it's easy to see that $x = q,s$ satisfies this equation.

One component of  $\gamma_{pqst}$ is the $x_1$-axis. If $p_1-2q_1=0$, $pqts$ is a rhombus. Otherwise
$$
x_2 \left[ x_2^2+ \left(x_1 - \frac{q_1^2+q_2^2}{p_1-2q_1}\right)^2  - \frac{(p_1-q_1)^2 + q_2^2}{(p_1-2q_1)^2}(q_1^2+q_2^2)\right] =0,
$$
and inside the square brackets we have an equation of a circle.

Observe that  by translation invariance, that changing $t=(0,0)$ to $t=(t_1,0)$ would merely effect the change $p_1\to p_1-t_1,\;q_1\to q_1-t_1,\;x_1\to x_1-t_1$ in the latter expression.

This completes the proof of the lemma.
\end{proof}

Recall Lemma \ref{l:two}.
\multiplicity*

 \begin{proof}[Proof of Lemma \ref{l:two}]  Without loss of generality we set $p=(1,0),\, q=(-1,0)$, so that the equation \eqref{e:wedgecurves} of the curve $\gamma_{pqst}$ becomes, with the notation $m=s+t,\,l=s-t$,
 $$
 \frac{|x|^2-1}{ 2x_2} = \frac{|x|^2 - x\cdot (s+t) + s\cdot  t}{s\wedge t - (s-t)\wedge x} = 
 \frac{|x|^2 - x_1m_1-x_2m_2 + \frac{|m|^2-|l|^2}{4}} {\frac{m\wedge l}{2} + l_1 x_2 - l_2x_1} \,.
 $$
 This yields a cubic (or lower degree) polynomial in $x$, with the projective vector of coefficients in the basis monomials $(x_1^3,x_1^2x_2,x_2^2x_1,x_2^3,x_1^2,x_1x_2,x_2^2, x_1, x_2,1)$ being
 $$
 \left(-l_2:l_1-2:-l_2:l_1-2:\frac{m\wedge l}{2}:2m_1:\frac{m\wedge l}{2}+2m_2:l_2:-l_1-\frac{|m|^2-|l|^2}{2}:-\frac{m\wedge l}{2}\right)\,.
 $$
 Thus, scaling some of the monomials for convenience, we have a map from $(m,l)$ into the five-dimensional projective subspace, as follows:
 $$ (m,l)\to (l_2:l_1-2:{m\wedge l}:m_1:{m\wedge l}+4m_2:2l_1+|m|^2-|l|^2)\,.$$ Let us show that it is not possible that an image has more than two pre-images $(s,t)$. The calculations below end up with solving quadratic equations in one variable $z$: either there are at most two roots or the number of solutions is infinite when the equation is vacuous. We will show that assuming the latter leads to a single trivial outcome $(s,t)=(p,q)$, which is a contradiction.
 
 First set $z=l_2$, assuming $z\neq 0$, when the  image vector equals $(1:a:b:c:d:e)$. Let us find conditions on $(a,b,c,d,e)$, which have an infinite number of pre-images $(m,l)$.
 We have
 $$
 l_1-2=az,\;\; {m\wedge l}=bz, \;\;m_1=cz,\;\;4m_2+ bz=dz,\;\;2az+4+|m|^2-|l|^2 = ez\,.
 $$
The second equation says
 $$
 bz = cz^2-(d-b)(az+2)z/4\,.
 $$
  This equation in $z$ has infinitely many (more than two) solutions only if 
 $$
 d= -b,\;\;c=-ab/2\,, d-b = -2b\,.
 $$
 So $m_1 = (-ab/2)z$, $m\wedge l = bz$, $m_2= -bz/2$.
 
For the last equation to have infinitely many solutions, 
the quadratic term in $z$ must vanish, thus
$$
(ab)^2/4+ b^2/4 = 1 + a^2\,.
$$
Therefore, $b=\pm 2$.

Taking $b=2$ we have $m=(-az,-z),\;l=(az+2,z)$. Hence, given $a\neq 0$, $s=(1,0)$ and $t=(-1-az,-z)$. For different values of $z$ the points $t(z)$ are clearly on a line. But $s=p$, and $p$ may have only one neighbour in $P_2$,  that is $q$.  Thus, the only option is $(s,t)=(p,q)$, a contradiction.

Taking $b=-2$, we have $m=(az,z)$, $l=(az+2,z)$, so $s=(1+az,z)$,  $t=(-1,0)$. Thus now $t=q$,  which yields the same conclusion as the above case $b=2$: $(s,t) = (p,q)$.

From now on we set $l_2=0$, in other words the vector $s-t$ is horizontal. 
Let us assume that $l_1\neq 2$, that is $(s,t)$ are mapped to the projective vector 
$$(l_1-2:-m_2l_1:m_1:4m_2-m_2l_1:2l_1 + |m|^2-l_1^2)=(1:a:b:c:d)\,.$$

In order for some $(a,b,c,d)$ to have more than two pre-images, the following system of equations is vacuous:
$$\begin{array}{rcl}l_1&=& z+2,\\ m_2(z+2) &=& -az,\\ m_1&=&bz,\\ 4m_2-m_2(z+2) &= & cz,\\ 2(z+2) + (bz)^2 - (z+2)^2 &= &dz\,.\end{array}$$
Thus $a=c=m_2=0$, namely the points $s,t$ lie on the horizontal coordinate axis. The last equation then implies that $b=\pm 1$, that is $m_1= \pm (l_1-2)$. Hence $s_1+t_1 = s_1-t_1 - 2$ which means $t=(-1,0)$ or $s_1+t_1 = -s_1+t_1 + 2$ or $s=(1,0)$. In both cases, since knowing, say $s$ defines the neighbour $t$, we arrive in the trivial solution $(s,t)=(p,q)$.

If we assume $l_2=2$ and $m\wedge l \neq 0$, that is $m_1\neq 0$, that $(s,t)$ will be mapped to the projective vector 
$(-2m_2: m_1: |m|^2).$ This vector clearly cannot infinitely many pre-images $(m_1,m_2)$. This is also the case when $m_1=0$.

This completes the proof of the lemma for the case of full curves $\gamma_{pqst}.$ To embrace the circular components in the reducible Scenario R2, we recall Lemma \ref{l:reducible} and its proof. Suppose, the centre of the circle is located on the line $pt$, with $q,s$ being symmetric relative to this line. Then given $t$ we know $q$, hence $s$. Moreover, changing $p$ would change the locus of the centre of the circle. Hence, the multiplicity of the circle is at most $2N$, as $(p,t)$ and $(s,q)$ can be swapped in the Scenario R2.

\end{proof}

Recall Lemma \ref{l:incidence}.
\incidenceBound*

\begin{proof}[Proof of Lemma \ref{l:incidence}]
    The proof uses Cauchy-Schwarz and the assumption that two distinct curves can intersect at most $C$ times. Rearranging and squaring the above definition of $|I(P,\Gamma)|$ gives
    \[|I(P,\Gamma)|^2 = \left(\sum_{p\in P} 1 \cdot \left(\sum_{\gamma\in \Gamma} m_\gamma \delta_{p\in \gamma}\right)\right)^2.\]
    Thus, by Cauchy-Schwarz we have that
    \[|I(P,\Gamma)|^2 \leq |P|\left(\left(\sum_{p\in P}\sum_{\gamma\in \Gamma}\sum_{\gamma'\in \Gamma\setminus\{\gamma\}}m_\gamma m_{\gamma'}\delta_{p\in \gamma}\delta_{p\in \gamma'}\right)+\left(\sum_{p\in P} \sum_{\gamma=\gamma'} m_\gamma^2\delta_{p\in \gamma}\right)\right).\]
    We will estimate each of these large summands separately. For the case when $\gamma\neq\gamma'$ we rearrange the order of summation to sum over the points in $P$ first.
    \begin{equation}\label{e:CauchySchwarz}
        \sum_{p\in P}\sum_{\gamma\in \Gamma}\sum_{\gamma'\in \Gamma\setminus\{\gamma\}}m_\gamma m_{\gamma'}\delta_{p\in \gamma}\delta_{p\in \gamma'} = \sum_{\gamma\in \Gamma}\sum_{\gamma'\in \Gamma\setminus\{\gamma\}}\sum_{p\in P}m_\gamma m_{\gamma'}\delta_{p\in \gamma}\delta_{p\in \gamma'}.
    \end{equation}
    Once we have fixed $\gamma$ and $\gamma'$ our assumption tells us that they can only cross in $C$ distinct points. Thus we have the bound
    \[\sum_{p\in P}m_\gamma m_{\gamma'}\delta_{p\in \gamma}\delta_{p\in \gamma'} \leq Cm_\gamma m_{\gamma'}.\]
    Putting the above into \eqref{e:CauchySchwarz} and separating out the sums gives
    \[\sum_{p\in P}\sum_{\gamma\in \Gamma}\sum_{\gamma'\in \Gamma\setminus\{\gamma\}}m_\gamma m_{\gamma'}\delta_{p\in \gamma}\delta_{p\in \gamma'}
    \leq 
    C\left(\sum_{\gamma}m_{\gamma}\right)\left(\sum_{\gamma'}m_{\gamma'}\right).
    \]
    Using the definition of $M$, we have that
    \begin{equation}\label{e:OffDiagonal}
    \sum_{p\in P}\sum_{\gamma\in \Gamma}\sum_{\gamma'\in \Gamma\setminus\{\gamma\}}m_\gamma m_{\gamma'}\delta_{p\in \gamma}\delta_{p\in \gamma'}
    \leq 
    CM^2.
    \end{equation}
    For the case when $\gamma=\gamma'$ we bound one of the $m_\gamma$ by $m_{\max}$ and note that we have recovered the definition of $|I(P,\Gamma)|$. Indeed,
    \begin{equation}\label{e:Diagonal}
        \sum_{p\in P} \sum_{\gamma\in\Gamma} m_\gamma^2\delta_{p\in \gamma} \leq m_{\max} \left(\sum_{p\in P} \sum_{\gamma\in\Gamma} m_\gamma\delta_{p\in \gamma}\right) = m_{\max}|I(P,\Gamma)|.
    \end{equation}
    Combining \eqref{e:OffDiagonal} and \eqref{e:Diagonal} gives us that
    \begin{equation}\label{e:quadBound}
        |I(P,\Gamma)|^2 \leq |P|\left( CM^2 + m_{\max}|I(P,\Gamma)|\right).
    \end{equation}
    One can use the quadratic formula at this stage to obtain an precise inequality at this stage, we avoid this as we only need an asymptotic bound.
    Suppose that \eqref{e:quadBound} is dominated by the first term, then
    \[|I(P,\Gamma)| \leq C^{1/2}M|P|^{1/2}.\]
    If \eqref{e:quadBound} is dominated by the second term, then
    \[|I(P,\Gamma)| \leq m_{\max}|P|.\]
    Combining the two bounds give the result.
\end{proof}

Recall Lemma \ref{lem:denseOnCircles}.
\Pups*

\begin{proof}[Proof of Lemma \ref{lem:denseOnCircles}]

We order the circles $C\in\mathcal{C}$ by non-increasing cardinality of their intersection with $P_1$, i.e. $\mathcal C=\{C_1,C_2,\ldots\},$ with  $|C_i\cap P_i| \geq |C_j\cap P_1|$ for $j>i$,  breaking ties ad hoc.

Let $k$ be minimal so that 
\[\Bigg| \bigcup_{i=1}^k (P_1 \cap C_i)\Bigg| \geq \frac12|P_1|.\]

If $k=1$, we are done, for either $50\%$ of $P_2$ is supported on $C_1$ (the second outcome of the lemma) or on $\mathcal C\setminus\{C_1\}$ (the first outcome, with $\mathcal C_1=\{C_1\}$, $\mathcal C_2=\mathcal C\setminus\mathcal C_1$).

Henceforth assume $k>1$.

If  \[\Bigg| \bigcup_{i>k} (P_2 \cap C_i)\Bigg| \geq \frac14|P_2|.\] we have the first outcome of the  lemma, with $\mathcal{C}_1 = \{C_1, \ldots, C_k\}$ and $\mathcal{C}_2 = \mathcal C\setminus\mathcal C_1.$

Otherwise, we reduce $k$ by one. Namely, by definition of $k$,
\[\Bigg| \bigcup_{i\geq k} (P_1 \cap C_i)\Bigg| \geq \frac12|P_1|.\]
Then if 
\[\Bigg| \bigcup_{i=1}^{k-1} (P_2 \cap C_i)\Bigg| \geq \frac12|P_2|\]
we are also done with the first outcome of the lemma, with $\mathcal C_1= \{C_i\}_{i\geq k}$ and $\mathcal C_2= \{C_i\}_{i< k}$

We are left with the case when more than $75\%$ of $P_2$ is supported on $\bigcup_{i=1}^k C_k$, with less than $50\%$ on $\bigcup_{i=1}^{k-1} C_k$, hence $|P_2\cap C_k|\geq \frac14|P_2|.$
However, from the ordering in question, 
\[\Bigg| \bigcup_{i=1}^{k-1} (P_1 \cap C_i)\Bigg| \geq \frac14|P_1|,\]
and we are once again done with the first outcome, with $\mathcal C_1= \{C_i\}_{i< k}$ and $\mathcal C_2=\{C_i\}_{i\geq k}$.

 \end{proof}

\bibliography{angles}{}

\bibliographystyle{plain}
 
\end{document}